\documentclass[a4paper,12pt]{amsart}

\usepackage{amsmath}
\usepackage{amssymb}
\usepackage{enumitem}
\usepackage{amsfonts}
\usepackage{graphicx}
\usepackage{mathtools}
\usepackage[colorlinks]{hyperref}
\renewcommand\eqref[1]{(\ref{#1})} 

\usepackage{mathrsfs}
\graphicspath{ {images/} }
\newcommand{\mfL}{\mathfrak{L}}
\newcommand{\mh}{\mathcal{H}}

\title[Heat and wave type equations with non-local operators, II.]{Heat and wave type equations with non-local operators, II. Hilbert spaces and graded Lie groups}

\author[M. Chatzakou]{Marianna Chatzakou}
\address{
	Marianna Chatzakou:
	\endgraf
	Department of Mathematics: Analysis, Logic and Discrete Mathematics
	\endgraf
	Ghent University, Krijgslaan 281, Building S8, B 9000 Ghent
	\endgraf
	Belgium
	\endgraf
	{\it E-mail address} {\rm marianna.chatzakou@ugent.be}}

\author[J. E. Restrepo]{Joel E. Restrepo}
\address{
	Joel E. Restrepo:
	\endgraf
	Department of Mathematics: Analysis, Logic and Discrete Mathematics
	\endgraf
	Ghent University, Krijgslaan 281, Building S8, B 9000 Ghent
	\endgraf
	Belgium
	\endgraf
	{\it E-mail address} {\rm joel.restrepo@ugent.be; cocojoel89@yahoo.es}}

\author[M. Ruzhansky]{Michael Ruzhansky}
\address{
	Michael Ruzhansky:
	\endgraf
	Department of Mathematics: Analysis, Logic and Discrete Mathematics
	\endgraf
	Ghent University, Krijgslaan 281, Building S8, B 9000 Ghent
	\endgraf
	Belgium
	\endgraf
	and
	\endgraf
	School of Mathematical Sciences
		\endgraf Queen Mary University of London 
			\endgraf
		United Kingdom
			\endgraf
	{\it E-mail address} {\rm michael.ruzhansky@ugent.be}}


\newtheoremstyle{theorem}
{10pt}          
{10pt}  
{\sl}  
{\parindent}     
{\bf}  
{. }    
{ }    
{}     
\theoremstyle{theorem}

\numberwithin{equation}{section}
\theoremstyle{plain}
\newtheorem{thm}{Theorem}[section]

\theoremstyle{definition}
\newtheorem{defn}[thm]{Definition}
\newtheorem{rem}[thm]{Remark}
\newtheorem{ex}[thm]{Example}

\newtheoremstyle{defi}
{10pt}          
{10pt}  
{\rm}  
{\parindent}     
{\bf}  
{. }    
{ }    
{}     
\theoremstyle{defi}

\begin{document}
 	\begin{abstract}
We study heat and wave type equations on a separable Hilbert space $\mathcal{H}$ by considering non-local operators in time with any positive densely defined linear operator with discrete spectrum. We show the explicit representation of the solution and analyse the time-decay rate in a scale of suitable Sobolev space. We perform similar analysis on multi-term heat and multi-wave type equations. The main tool here is the Fourier analysis which can be developed in a separable Hilbert space based on the linear operator involved. As an application, the same Cauchy problems are considered and analysed in the setting of a graded Lie group. In this case our analysis relies on the group Fourier analysis. An extra ingredient in this framework allows, in the case of heat type equations, to establish $L^p$-$L^q$ estimates for $1\leqslant p\leqslant 2\leqslant q<+\infty$ for the solutions on graded Lie group groups. Examples and applications of the developed theory are given, either in terms of self-adjoint operators on compact or non-compact manifolds, or in the case of particular settings of graded Lie groups. The results of this paper significantly extend in different directions the results of Part I (\cite{WRR}), where operators on compact Lie groups were considered. We note that the results obtained in this paper are also new already in the Euclidean setting of $\mathbb{R}^n$. 
	\end{abstract}
	\maketitle
	\tableofcontents

\section{Introduction}

Evolution integral equations have been studied by many authors due to their applications in different fields, see e.g. the entire monographs \cite{mainardi,pruss,tricomi}. This branch can model problems as the theory of viscoelastic material behaviour. We can have simple shearing motions, torsion of a rod, simple tension, etc. Also, it is related with viscoelastic fluids and beams, heat conduction with memory, electrodynamics with memory, etc. Here we consider integro-differential operators in time which kernels are $\frac{t^{-\beta}}{\Gamma(1-\beta)}$ $(0<\beta\leqslant 1)$ or $\frac{t^{1-\beta}}{\Gamma(2-\beta)}$ $(1<\beta<2)$. In fact, we will use non-local operators (in time) of the form
\[
\int_0^t \frac{(t-s)^{n-\beta-1}}{\Gamma(n-\beta)}\partial_{s}^{(n)}w(x,s)\,\mathrm{d}s,\quad 0<\beta<2\,\,(\beta\neq1),\quad n=\lfloor\beta\rfloor+1, 
\]
to investigate heat and wave type equations. For the range $0<\beta<1$, the equation is modelling the fractional relaxation, while for $1<\beta<2$, the equation is refereed to as  the fractional oscillation of the system. A classical work on this direction was given by Riesz in \cite{Riesz} where he studied the Cauchy problem by the Riemann-Liouville fractional operators. Contemporary investigations on Cauchy problems, diffusion and wave equations, etc., can be found e.g. in \cite{Fujita1,Fujita2,memory1,memory2,FractionalDiffusion}. These preliminary studies were principally given on the real axis. On $\mathbb{R}^n$, we can find several works attempting questions on the well-posedness, regularity and decay rate behaviour of the solutions, see e.g. \cite{para,cauchy,uno2,uno3,uno4,uno5}. 

Notice that the studies of fractional evolution equations are becoming more popular in the last 20 years since it has shown to have an intrinsic nature and good properties which are difficult to realise  from the general theory of integral equations. Let us mention the following works \cite{thesis2001,thesis,hilbert,section3} and the references therein. Some of the works have been performed in very general settings like Banach or Hilbert spaces. Each space has their own advantages, disadvantages and inherent features. 

In this paper we focus on analysing heat and wave type equations by using integro-differential operators (in time) on a separable Hilbert space. Here we exploit the Fourier analysis which can be obtained from the consideration of a linear operator densely defined with discrete spectrum \cite{[15],fourier-1,[51]-1}, see also \cite{RNT}. In particular, we use the Fourier analysis of a graded Lie group to show the applicability of the general results. The first step in this direction was given recently in \cite{RNT}. In the last cited paper, we can see that the considered equations are restricted to some boundary conditions, even in time where it is conditioned to $0<t\leqslant T<+\infty.$ Now we provide a more general type of equations without such restrictions that allows to investigative the time decay rate in all the cases for any $t\geqslant 0$, excluding just the case of multi-term type equations where $t<T$. This paper is a continuation of the results given on compact Lie groups in \cite{WRR}. One of the new ingredients of our analysis is the fact that the multivariate Mittag-Leffler function emerged naturally as the inverse Laplace transform of a function, see \cite{new-mittag-add,17} and also \cite{laplace-inverse}.          

\medskip Now we give a brief description of the main results of this article. 

We begin Section \ref{preli} by recalling some fundamentals aspects of fractional evolution equation on Banach spaces. We also recall the Fourier analysis on  Hilbert spaces and on  graded Lie groups. These elements will be used in the whole paper. 

In Section \ref{explicit}, we give the main results of this manuscript about the explicit representation and time decay rate of the solutions of the considered heat and wave type equations. In particular, we establish the following statements.

Below we always consider $\mfL:\text{Dom}(\mfL)\subset\mathcal{H}\to\mathcal{H}$ to be a positive linear operator densely defined with discrete spectrum $\sigma(\mfL)=\{\gamma_\zeta:\zeta\in I\}$ ($I$ is a countable set of indices) on a separable Hilbert space $\mathcal{H}$, with the system of eigenfunctions $\{e_\zeta\}_{\zeta\in I}$ of $\mfL$ an orthonormal basis in $\mathcal{H}$. Thus, we can consider e.g. Bessel or differential operators, and also problems with respect to Sturm-Liouville, harmonic oscillator, anharmonic oscillator, Landau Hamiltonian, etc. See \cite{[15],fourier-1,RNT} and the examples in Section \ref{examples} for more details. 

Notice that compact Lie groups under the consideration of operators like Laplacian or sub-Laplacian will fit as a particular case of this analysis. Here we can recover some of the results of \cite{WRR}. Moreover, our results can also involve compact and non-compact manifolds \cite{[15]}.

We also use a Sobolev space $\mathcal{H}_{\mathfrak{L}}^{\delta}$ $(\delta\in\mathbb{R})$ defined in Subsection \ref{FA-hilbert}. More details can be found there. We note that, in the representation of the solution,  the well-known Mittag-Leffler function appears, see e.g. the entire book dedicated to the study of different types of Mittag-Leffler functions, its properties, etc \cite{mittag}.

\medskip Let us consider the {\bf $\mfL$-heat type equation} 
\begin{equation}\label{heat-FDE-intro}
\prescript{C}{0}{\partial}^{\beta}w(t)+\mathfrak{L}w(t)=0,\qquad t>0,\quad 0<\beta\leqslant 1,
\end{equation}
where $\prescript{C}{0}{\partial}^{\beta}$ is the Dzhrbashyan-Caputo fractional derivative as in \eqref{djerbashian}, and the operator $\mfL$ is considered to be positive,
under the initial condition
\[
w(t)|_{t=0}=w_0\in\mathcal{H}.
\] 
For the definition of the Sobolev spaces $\mathcal{H}_{\mathfrak{L}}^{\delta}$, $\delta \in \mathbb{R}$, which appear in our first result given below, we refer to Subsection \ref{FA-hilbert}. 
\begin{thm}
Let $0<\beta\leqslant1$ and $\delta\in\mathbb{R}$. 
\begin{enumerate}
    \item If $0\in\sigma(\mfL)$ and $w_0\in\mathcal{H}_{\mathfrak{L}}^{\delta}$ then there exists a unique continuous solution $w(t)\in \mathcal{H}^{\delta}_{\mfL}$ for any $t\in(0,+\infty)$ of equation \eqref{heat-FDE-intro} given by
\begin{equation}\label{solution-heat-intro}
w(t)=\sum_{\zeta\in I}(w_0,e_\zeta)E_{\beta}(-\gamma_\zeta t^{\beta})e_{\zeta},
\end{equation}
that satisfies the following estimate
\[
\|w(t)\|_{\mathcal{H}^{\delta}_{\mathfrak{L}}}\lesssim \|w_0\|_{\mathcal{H}_{\mathfrak{L}}^{\delta}},\quad t\geqslant0.
\]
\item If $0\notin\sigma(\mfL)$ and $w_0\in\mathcal{H}_{\mathfrak{L}}^{\delta}$ then there exists a unique continuous solution $w(t)\in \mathcal{H}^{\delta}_{\mfL}$ for any $t\in(0,+\infty)$ of equation \eqref{heat-FDE-intro} represented by \eqref{solution-heat-intro} 
such that 
\[
\|w(t)\|_{\mathcal{H}^{\delta}_{\mathfrak{L}}}\lesssim (1+t^{\beta})^{-1}\|w_0\|_{\mathcal{H}_{\mathfrak{L}}^{\delta}},\quad t\geqslant0.
\]
\item If $w_0\in\mathcal{H}_{\mathfrak{L}}^{\delta-2}$ then there exists a unique continuous solution $w(t)\in \mathcal{H}^{\delta}_{\mfL}$ for any $t\in(0,+\infty)$ of equation \eqref{heat-FDE-intro} given by \eqref{solution-heat-intro} and we have
\[
\|w(t)\|_{\mathcal{H}^{\delta}_{\mathfrak{L}}}\lesssim (1+t^{-\beta})\|w_0\|_{\mathcal{H}_{\mathfrak{L}}^{\delta-2}},\quad t>0.
\]
\end{enumerate}
\end{thm}

The next result is on the solution to the {\bf $\mfL$-wave type equation} that reads as follows
\begin{equation}\label{wave-FDE-intro}
\prescript{C}{0}{\partial}^{\beta}w(t)+\mathfrak{L}w(t)=0,\qquad t>0,\quad 1<\beta<2,
\end{equation}
under the initial conditions
\[
\partial_t^{(k)}w(t)|_{t=0 }=w_k\in\mathcal{H},\quad k=0,1\,.
\]

\begin{thm}
Let $1<\beta<2$ and $\delta\in\mathbb{R}$. 
\begin{enumerate}
    \item If $0\in\sigma(\mfL)$ and $w_0,w_1\in\mathcal{H}_{\mathfrak{L}}^{\delta}$ then there exists a unique continuous solution $w(t)\in \mathcal{H}^{\delta}_{\mfL}$ for any $t\in(0,+\infty)$ of equation \eqref{wave-FDE-intro} represented by
\begin{equation}\label{solution-wave-intro}
w(t)=\sum_{\zeta\in I}\Bigg[(w_0,e_\zeta)_{\mathcal{H}}E_{\beta}(-t^{\beta}\gamma_\zeta)+t(w_1,e_\zeta)_{\mathcal{H}}E_{\beta,2}(-t^{\beta}\gamma_\zeta)\Bigg]e_{\zeta},
\end{equation}
and we have
\[
\|w(t)\|_{\mathcal{H}^{\delta}_{\mathfrak{L}}}\lesssim \|w_0\|_{\mathcal{H}_{\mathfrak{L}}^{\delta}}+t\|w_1\|_{\mathcal{H}_{\mathfrak{L}}^{\delta}},\quad t\geqslant0.
\]
\item If $0\notin\sigma(\mfL)$ and $w_0,w_1\in\mathcal{H}_{\mathfrak{L}}^{\delta}$ then there exists a unique continuous solution $w(t)\in \mathcal{H}^{\delta}_{\mfL}$ for any $t\in(0,+\infty)$ of equation \eqref{wave-FDE-intro} given by \eqref{solution-wave-intro} with \[
\|w(t)\|_{\mathcal{H}^{\delta}_{\mathfrak{L}}}\lesssim (1+t^{\beta})^{-1}\|w_0\|_{\mathcal{H}_{\mathfrak{L}}^{\delta}}+t(1+t^{\beta})^{-1}\|w_1\|_{\mathcal{H}_{\mathfrak{L}}^{\delta}},\quad t\geqslant0.
\]
\item If $(w_0,w_1)\in\bigg(\mathcal{H}_{\mathfrak{L}}^{\delta},\mathcal{H}_{\mathfrak{L}}^{\delta-\frac{2}{\beta}}\bigg)$ then there exists a unique continuous solution $w(t)\in \mathcal{H}^{\delta}_{\mfL}$ for any $t\in(0,+\infty)$ of equation \eqref{wave-FDE-intro} represented by \eqref{solution-wave-intro} such that 
\[
\|w(t)\|_{\mathcal{H}^{\delta}_{\mathfrak{L}}}\lesssim \|w_0\|_{\mathcal{H}_{\mathfrak{L}}^{\delta}}+(1+t)\|w_1\|_{\mathcal{H}_{\mathfrak{L}}^{\delta-\frac{2}{\beta}}},\quad t\geqslant0.
\]
\item If $(w_0,w_1)\in\big(\mathcal{H}_{\mathfrak{L}}^{\delta},\mathcal{H}_{\mathfrak{L}}^{\delta-2}\big)$ then there exists a unique continuous solution $w(t)\in \mathcal{H}^{\delta}_{\mfL}$ for any $t\in(0,+\infty)$ of equation \eqref{wave-FDE-intro} given by \eqref{solution-wave-intro} and we have 
\[
\|w(t)\|_{\mathcal{H}^{\delta}_{\mathfrak{L}}}\lesssim \|w_0\|_{\mathcal{H}_{\mathfrak{L}}^{\delta}}+t(1+t^{-\beta})\|w_1\|_{\mathcal{H}_{\mathfrak{L}}^{\delta-2}},\quad t>0.
\]
\end{enumerate}
\end{thm}

We also study the {\bf multi-term $\mfL$-wave-heat type equation}:
 \begin{equation}\label{multi-wave-intro}
\left\{ \begin{aligned}
\prescript{C}{}\partial_{t}^{\beta}w(t)+\sigma_1\prescript{C}{}\partial_{t}^{\beta_1}w(t)+\cdots+\sigma_m\prescript{C}{}\partial_{t}^{\beta_m}w(t)+\mfL w(t)&=0,  \\
\partial_t^{(k)}w(t)\big|_{t=0^+}\big.&=w_k\in\mathcal{H},\quad k=0,1,
\end{aligned}
\right.
\end{equation}
for $0<t\leqslant T<+\infty$ (this restriction is associated with the boundedness of the solution), where $\sigma_i\geqslant0$ $(i=1,\ldots,m)$ and $2>\beta>\beta_1>\cdots>\beta_{m}>0$. Here we have two cases. First, if $\beta\in(0,1]$ then we consider just the initial condition $w_0.$ The second one, if $\beta\in(1,2)$, we then use both initial conditions $w_0,w_1\in\mathcal{H}.$   
\begin{thm}
Let $\delta\in\mathbb{R}$ and $2>\beta>\beta_1>\cdots>\beta_{m}>0$. 
\begin{enumerate}
    \item If $0\in\sigma(\mfL)$ and $w_0,w_1\in \mathcal{H}^{\delta}_{\mfL}$ then there exists a unique solution $w(t)\in \mathcal{H}^{\delta}_{\mathfrak{L}}$ for any $t\in(0,T]$ for the equation \eqref{multi-wave-intro} given by
\begin{align}\label{solution-wave-multi-intro}
w(t)&=\sum_{k=0}^{m}\sigma_k t^{\beta-\beta_k}\sum_{\zeta\in I}(w_0,e_\zeta)_{\mathcal{H}}E_{(\beta-\beta_1,\ldots,\beta-\beta_m,\beta),\beta-\beta_k+1}(-\sigma_1 t^{\beta-\beta_1},\ldots \\
&\hspace{8cm}\ldots,-\sigma_m t^{\beta-\beta_m},-\gamma_\zeta t^{\beta}), \nonumber \\
&+\sum_{k=0}^{m}\sigma_k t^{\beta-\beta_k+1}\sum_{\zeta\in I}(w_1,e_\zeta)_{\mathcal{H}}E_{(\beta-\beta_1,\ldots,\beta-\beta_m,\beta),\beta-\beta_k+2}(-\sigma_1 t^{\beta-\beta_1},\ldots \nonumber \\
&\hspace{8cm}\ldots,-\sigma_m t^{\beta-\beta_m},-\gamma_\zeta t^{\beta}), \nonumber
\end{align}
where $\sigma_0=1, \beta_0=\beta$ and it follows that 
\begin{align*}
\|w(t)\|_{\mathcal{H}_\mfL^{\delta}}\leqslant C_{T,\vec{\beta},\vec{\sigma}} \left(\sum_{k=0}^{m}\sigma_k t^{\beta-\beta_k}\right)\big(\|w_0\|_{\mathcal{H}_\mfL^{\delta}}+t\|w_1\|_{\mathcal{H}_\mfL^{\delta}}\big),  
\end{align*}
for some constant which depends on $T$, $\vec{\beta}=(\beta,\ldots,\beta_m)$ and $\vec{\sigma}=(\sigma_0,\ldots,\sigma_m).$

\item If $0\notin\sigma(\mfL)$ and $w_0,w_1\in \mathcal{H}^{\delta}_{\mfL}$ then there exists a unique solution $w(t)\in \mathcal{H}^{\delta}_{\mathfrak{L}}$ for any $t\in(0,T]$ for the equation \eqref{multi-wave-intro} given by \eqref{solution-wave-multi-intro} and we have 
\begin{align*}
\|w(t)\|_{\mathcal{H}_\mfL^{\delta}}\leqslant C_{T,\vec{\beta},\vec{\sigma}} \left(\sum_{k=0}^{m}\sigma_k t^{\beta-\beta_k}\right)(1+t^{\beta})^{-1}\big(\|w_0\|_{\mathcal{H}_\mfL^{\delta}}+t\|w_1\|_{\mathcal{H}_\mfL^{\delta}}\big). 
\end{align*}
\item If $(w_0,w_1)\in\bigg( \mathcal{H}^{\delta}_{\mfL}\times \mathcal{H}^{\delta-\frac{2}{\beta}}_{\mfL}\bigg)$ then there exists a unique solution $w(t)\in \mathcal{H}^{\delta}_{\mathfrak{L}}$ for any $t\in(0,T]$ for the equation \eqref{multi-wave-intro} given by \eqref{solution-wave-multi-intro} and we have 
\begin{align*}
\|w(t)\|_{\mathcal{H}_\mfL^{\delta}}\leqslant C_{T,\vec{\beta},\vec{\sigma}}\left(\sum_{k=0}^{m}\sigma_k t^{\beta-\beta_k}\right)\bigg(\|w_0\|_{\mathcal{H}_\mfL^{\delta}}+(1+t)\|w_1\|_{\mathcal{H}_{\mfL}^{\delta-\frac{2}{\beta}}}\bigg).
\end{align*}
\item If $w_0,w_1\in \mathcal{H}^{\delta-2}_{\mfL}$ then there exists a unique solution $w(t)\in \mathcal{H}^{\delta}_{\mathfrak{L}}$ for any $t\in(0,T]$ for the equation \eqref{multi-wave-intro} given by \eqref{solution-wave-multi-intro} and we have 
\begin{align*}
\|w(t)\|_{\mathcal{H}_\mfL^{\delta}}\leqslant C_{T,\vec{\beta},\vec{\sigma}}(1+t^{-\beta})\left(\sum_{k=0}^{m}\sigma_k t^{\beta-\beta_k}\right)\big(\|w_0\|_{\mathcal{H}_\mfL^{\delta-2}}+t\|w_1\|_{\mathcal{H}_\mfL^{\delta-2}}\big).
\end{align*}
\end{enumerate}
\end{thm}

We conclude the paper with Section \ref{ss-graded}. Here we show results analogous to  Section \ref{explicit} in the case of {\bf graded Lie groups} $G$. The setting of graded Lie groups includes the Euclidean setting $\mathbb{R}^n$, the Heisenberg group $\mathbb{H}^n$, and more generally any stratified group. In this case, the natural choice of operators is the Rockland ones, and the spectrum of such operators on the Hilbert space $L^2(G)$ is continuous. However, using the Fourier analysis on the group, the analysis there can be reduced to the previous analysis with discrete spectrum. Moreover, for heat type equations, we are able to provide $L^p-L^q$ estimates for the solution, which is a different feature that the current setting allows to obtain. Particularly, this is a consequence of some recent results on the $L^p-L^q$ boundedness of  Fourier multipliers on locally compact groups \cite{RR2020}. Briefly, the results are as follows.   
  
Below we consider $\mathcal{R}$ to be a positive Rockland operator of homogeneous degree $\nu$, and for $s\in\mathbb{R}$ the notation $\dot{L}^{2}_{s}(G)$ (resp. $L^{2}_{s}(G)$) stands for the homogeneous (resp. nonhomogeneous) Sobolev spaces on $G$ as introduced in subsection  \ref{sub.int.graded}. We study the following {\bf $\mathcal{R}$-wave type equation}:  
\begin{equation}\label{r-wave-intro}
\left\{ \begin{aligned}
\prescript{C}{}\partial_{t}^{\beta}w(t,x)+\mathcal{R}w(t,x)&=0,\quad t>0,\quad x \in G,\quad 1<\beta<2,  \\
w(t,x)|_{_{_{t=0}}}&=w_0(x)\,, \\
\partial_t w(t,x)|_{_{_{t=0}}}&=w_1(x)\,.
\end{aligned}
\right.
\end{equation}
  We refer to Subsection \ref{sub.int.graded} for relevant details of the Fourier analysis on graded Lie groups.
\begin{thm} Let $\mathcal{R}$ be a positive Rockland operator of homogeneous degree $\nu$ on the graded Lie group $G$. On $G$ we consider the Cauchy problem \eqref{r-wave-intro}.   
    We have:
   
    \begin{enumerate}[label=(\alph*)]
        \item for any $t>0$ the solution to the problem \eqref{r-wave-intro} is explicitly given by 
        \[
w(t,x)=E_\beta(-t^{\beta}\mathcal{R})w_0(x)+tE_{\beta,2}(-t^{\beta}\mathcal{R})w_1(x)\,,\quad x \in G\,;
        \]
\item  if $(w_0,w_1) \in L^2(G)\times L^2(G)$, then the solution $w$ is unique and satisfies the estimate
\[
    \|w(t,\cdot)\|_{L^2(G)} \lesssim \|w_0\|_{L^2(G)}+t\|w_1\|_{L^2(G)}\,,\quad \text{for all} \quad t>0;
\]
\item if $(w_0,w_1) \in L^{2}_{s}(G)\times L^{2}_{s}(G)$, then the solution $w$ is unique and satisfies the Sobolev-norm estimate
\[
\|w(t,\cdot)\|_{L^{2}_{s}(G)}\lesssim \|w_0\|_{L^{2}_{s}(G)}+t\|w_1\|_{L^{2}_{s}(G)}\,,\quad{\text{for any}}\quad s\in\mathbb{R},
\]
which, in particular, yields the estimate in \ref{itm:2est} for $s=0$;
\item if $(w_0,w_1) \in L^{2}_{s}(G)\times L^{2}_{s-\nu/\beta}(G)$, then the solution $w$ is unique and satisfies the Sobolev-norm estimate
\[
    \|w(t,\cdot)\|_{L^{2}_{s}(G)}\lesssim \|w_0\|_{L^{2}_{s}(G)}+ (1+t)\|w_1\|_{L^{2}_{s-\nu/\beta}(G)},\quad s\geqslant \frac{\nu}{\beta}.
\]
\item Finally, we also obtain that 
\begin{align*}
     \|\partial_t w(t,\cdot)\|_{L^2(G)}&\lesssim 
\left\{
\begin{array}{rccl}
& t^{\beta-1}\|w_0\|_{\dot{L}_{\nu}^2(G)}+\|w_1\|_{L^2(G)},\,\quad (w_0,w_1)\in \dot{L}_{\nu}^2(G)\times L^2(G), \\
& \|w_0\|_{\dot{L}_{\nu/\beta}^2(G)}+\|w_1\|_{L^2(G)},\,\quad (w_0,w_1)\in \dot{L}_{\nu/\beta}^2(G)\times L^2(G), \\
& t^{-1}\|w_0\|_{L^2(G)}+\|w_1\|_{L^2(G)},\,\quad (w_0,w_1)\in L^2(G)\times L^2(G).
\end{array}
\right.     
 \end{align*}

    \end{enumerate}
\end{thm}

Let us now discuss the following {\bf $F(\mathcal{R})$-heat type equation}:

\begin{equation}\label{r-heat-intro}
\begin{split}
^{C}\partial_{t}^{\alpha}w(t,x)+F(\mathcal{R})w(t,x)&=0, \quad t>0,\,\, x\in G, \\
w(t,x)|_{_{_{t=0}}}&=w_0(x),
\end{split}
\end{equation}
where $F:[0,\infty) \rightarrow [0,\infty)$ is an increasing function such that $\displaystyle\lim_{s\to+\infty}F(s)=+\infty$ and $0<\alpha\leqslant 1$. 

Before stating our result on $L^p$-$L^q$ estimates for solutions of \eqref{r-heat-intro} let us introduce some necessary notation: We have denoted by $E_{(0,s)}(\mathcal{R})$  the spectral projections of the Rockland operator $\mathcal{R}$ to the interval $(0,s)$, and by $\tau$  the canonical trace on the right group von Neumann algebra $VN_{R}(G)$, see \cite{RR2020} for details.

\begin{thm}\label{thm1.5}
Let $0<\alpha\leqslant 1$ and $1\leqslant p\leqslant 2\leqslant q<+\infty$. Then there exists a unique continuous solution to the $F(\mathcal{R})$-heat type equation \eqref{r-heat-intro} represented by 
\[
w(t,x)=E_\alpha(-t^{\alpha}F(\mathcal{R}))w_0(x),\quad t>0,\,\,x\in G\,.
\]
Moreover, if $w_0 \in L^p(G)$, and 
\begin{align}\label{need-intro}
\sup_{t>0}\sup_{s>0}[\tau\big(E_{(0,s)}(\mathcal{R})\big)]^{\frac{1}{p}-\frac{1}{q}}E_\alpha(-t^{\alpha}F(s))<+\infty, 
\end{align}
then there exits a unique solution $w\in\mathcal{C}\big([0,+\infty);L^q(G)\big).$ 

In particular, if for some $\gamma>0$ we have 
\[
\tau\big(E_{(0,s)}(\mathcal{R})\big)\lesssim s^{\gamma},\quad s\to+\infty\,,\quad \text{and}\quad F(s)=s,
\]
then the condition \eqref{need-intro} is satisfied for any $1<p\leqslant 2\leqslant q<+\infty$ such that $\frac{1}{\gamma}>\frac{1}{p}-\frac{1}{q}$, and we get the following time decay rate for the solution of equation \eqref{r-heat-intro}: 
\[
\|w(t,\cdot)\|_{L^q(G)}\leqslant C_{\alpha,\gamma,p,q}t^{-\alpha\gamma\left(\frac{1}{p}-\frac{1}{q}\right)}\|w_0\|_{L^p(G)},    
\]
where $C_{\alpha,\gamma,p,q}$ does not depend on $w_0$ and $t>0.$
\end{thm}
Applications of the above theorem are given in Subsection \ref{heat-section}.

\medskip We conclude this work by studying the following {\bf multi-term heat-wave type equation}:
\begin{equation}\label{multi-heawave-intro}
	\left\{ \begin{aligned}
		\prescript{C}{}\partial_{t}^{\alpha_0}w(t,x)+c_1\prescript{C}{}\partial_{t}^{\alpha_1}w(t,x)+\cdots+c_m\prescript{C}{}\partial_{t}^{\alpha_m}w(t,x)+\mathcal{R}w(t,x)&=0,\,\,  \\
		w(t,x)|_{_{_{t=0}}}&=w_0(x), \\
  \partial_t w(t,x)|_{_{_{t=0}}}&=w_1(x)\,,
	\end{aligned}
	\right.
\end{equation}
for $0<t\leqslant T<+\infty$ and $x\in G$, where $c_i>0$ $(i=1,\ldots,m)$ and $0<\alpha_m<\alpha_{m-1}<\cdots<\alpha_1<\alpha_0<2.$

\begin{thm}
Let $0<\alpha_m<\alpha_{m-1}<\cdots<\alpha_1<\alpha_0<2$. Then there exists a unique continuous solution to the mutli-term wave equation in \eqref{multi-heawave-intro} given by
\begin{align*}
&w(t,x)=\sum_{k=0}^{m}t^{\alpha_0-\alpha_k}E_{(\alpha_0-\alpha_1,\ldots,\alpha_0-\alpha_m,\alpha_0),\alpha_0-\alpha_k+1}(-c_1 t^{\alpha_0-\alpha_1},\ldots,-c_m t^{\alpha_0-\alpha_m},-t^{\alpha_0}\mathcal{R}) w_0(x) \\
&+\sum_{k=0}^{m}t^{\alpha_0-\alpha_k+1}E_{(\alpha_0-\alpha_1,\ldots,\alpha_0-\alpha_m,\alpha_0),\alpha_0-\alpha_k+2}(-c_1 t^{\alpha_0-\alpha_1},\ldots,-c_m t^{\alpha_0-\alpha_m},-t^{\alpha_0}\mathcal{R})w_1(x),
\end{align*}
for any $0<t\leqslant T$ and $x\in G.$  

 Moreover we have the following Sobolev norm estimates:
\begin{enumerate}
    \item For any $s\in\mathbb{R}$ and $w_0,w_1\in L^{2}_{s}(G)$ we have 
     \[
     \|w(t,\cdot)\|_{L^{2}_{s}(G)}\leqslant C_{T,\vec{\alpha},\vec{\gamma}}\left(\sum_{k=0}^{m}\gamma_k t^{\alpha_0-\alpha_k}\right)\left(\|w_0\|_{L^{2}_{s}(G)}+t\|w_1\|_{L^{2}_{s}(G)}\right)\,,
     \]
     for $0<t\leqslant T$.
     \item For any $s\geqslant \frac{\nu}{\alpha_0}$ and for $(w_0,w_1)\in L^{2}_{s}(G)\times L^{2}_{s-\nu/\alpha_0}(G)$ we have
     \[
    \|w(t,\cdot)\|_{L^{2}_{s}(G)}\leqslant C_{T,\vec{\alpha},\vec{\gamma}}\left(\sum_{k=0}^{m}\gamma_k t^{\alpha_0-\alpha_k}\right)\left(\|w_0\|_{L^{2}_{s}(G)}+(1+t) \|w_1\|_{L^{2}_{s-\nu/\alpha_0}(G)}\right)\,,
     \]
      for $0<t\leqslant T$.
     \item For any $s \geqslant \nu$ and for $w_0,w_1 \in L^{2}_{s-\nu}(G)$ we have
     \[
       \|w(t,\cdot)\|_{L^{2}_{s}(G)}\leqslant C_{T,\vec{\alpha},\vec{\gamma}}\left(\sum_{k=0}^{m}\gamma_k t^{\alpha_0-\alpha_k}\right)(1+t^{-\alpha_0})\left(\|w_0\|_{L^{2}_{s-\nu}(G)}+t\|w_1\|_{L^{2}_{s-\nu}(G)}\right)\,,
     \]
      for $0<t\leqslant T$.
\end{enumerate}
Under the condition that $0<\alpha_m<\alpha_{m-1}<\cdots<\alpha_1<\alpha_0\leqslant1$, the analogous results in the case of the mutli-term heat equation in  \eqref{multi-heawave-intro} are as above if one considers $w_1$ to be identically zero.
\end{thm}

\section{Preliminary Notions}\label{preli}

We start by recalling some necessary definitions and basic results on integro-differential operators, abstract differential equations, Fourier analysis on a separable Hilbert space and graded Lie group, which will be used constantly in the development of this paper. 

Let us first recall some standard function spaces:
\begin{align*}
L^1(a,T)&=\left\{f:(a,T)\to\mathbb{R}\;:\;\big\|f\big\|_{L^1(a,T)}:=\int_a^T\big|f(t)\big|\,\mathrm{d}t<\infty\right\}; \\
AC[a,T]&=\left\{f:[a,T]\to\mathbb{R}\;:\;f\text{ is absolutely continuous on }[a,T]\right\}; \\
AC^n[a,T]&=\left\{f:[a,T]\to\mathbb{R}\;:\;f^{(n-1)}\text{ exists and is in }AC[a,T]\right\},\qquad n=1,2,\ldots,
\end{align*}
where $[a,T]\subseteq\mathbb{R}$ is a fixed finite interval. 

\subsection{Integro-differential operators}

The \textit{Riemann--Liouville fractional integral} of order $\beta>0$ is defined by \cite{kilbas} \cite[Sections 2.3 and 2.4]{samko}:
\[
\prescript{RL}{a}I^{\beta}f(t)=\frac1{\Gamma(\beta)}\int_a^t (t-s)^{\beta-1}f(s)\,\mathrm{d}s,\qquad f\in L^1(a,T),
\]
while the \textit{Riemann--Liouville fractional derivative} of order $\beta\geqslant0$ is given by:
\begin{align*}
\prescript{RL}{a}D^{\beta}f(t)&=D^{n}\prescript{RL}{a}I^{n-\beta}f(t) \\
&=\frac1{\Gamma(n-\beta)}\left(\frac{\rm{d}}{\rm{d}t}\right)^{n}\int_a^t (t-s)^{n-\beta-1}f(s)\,\mathrm{d}s,\,\, f\in AC^n[a,T].
\end{align*}
In this article, we use a non-local differential operator well-known as the \textit{Dzhrbashyan--Caputo fractional derivative}:
\begin{equation}\label{djerbashian}
\prescript{C}{a}D^{\beta}f(t)=\prescript{RL}{a}I^{n-\beta}f^{(n)}(t),\qquad f\in AC^n[a,T],\quad n:=\lfloor\beta\rfloor+1.
\end{equation}
The operator above can be rewritten in terms of the initial conditions of the function and by using the Riemann-Liouville fractional derivative in the following form:
\begin{equation}\label{caputo-alternative}
\prescript{C}{a}D^{\beta}f(t)=\prescript{RL}{a}D^{\beta}\left(f(t)-\sum_{k=0}^{n-1}\frac{f^{(k)}(a)}{k!}(t-a)^k\right),\qquad f\in AC^n[a,T],
\end{equation}
Notice that expression \eqref{caputo-alternative} holds for all functions $f\in AC^n[a,T]$, see e.g. \cite[Theorem 2.2]{samko}. Let us point out that even though  definitions \eqref{djerbashian} and \eqref{caputo-alternative} coincide on the function space 
 $AC^n[a,T]$, definition \eqref{caputo-alternative} is broader due to the fact that Riemann--Liouville derivative, and consequently the Dzhrbashyan--Caputo fractional derivative, can be defined on a larger function space. We refer to \cite{example} for a concrete example of the latter situation.

\subsection{Abstract fractional equations}\label{abstractsection}
Let us now recall some important notions related to the solution of an evolution integral equation. For an extensive study and basic results on this subject in a general form, we refer the reader to the monographs \cite{pruss,tricomi}. The books \cite{mainardi,samko} can be used as classical ones on the subject of fractional calculus. Also, for more specific information on abstract fractional equations, see \cite{thesis2001} and references therein.    

To begin with, let $\mathfrak{L}$ be a closed linear operator densely defined on a Banach space $X$. Suppose also that $\beta>0$ and $n=1+\lfloor \beta\rfloor$. We are interested in studying the following Cauchy problem involving the Dzhrbashyan-Caputo fractional derivative of order $\beta$: 
\begin{equation}\label{abstracte}
\prescript{C}{0}\partial^{\beta}_t w(t)=\mathfrak{L}w(t),\quad t>0;\quad w^{(j)}(0)=w_j,\quad j=0,1,\ldots,n-1. 
\end{equation}
\begin{defn}
Let $w$ be a function of the space $C(\mathbb{R}^{+},X)$. It is said that $w$ is a \textit{strong solution} of the equation \eqref{abstracte} if $w\in C(\mathbb{R}^{+},X)\cap C^{n-1}(\mathbb{R}^{+},X)$, 
\[
\prescript{RL}{0}I^{n-\beta}\left(w(t)-\sum_{j=0}^{n-1}\frac{w^{(j)}(0)}{j!}t^j\right)\in C^{n-1}(\mathbb{R}^{+},X)
\]
and equation \eqref{abstracte} is satisfied on $\mathbb{R}^+.$
\end{defn}

\begin{defn}
The Cauchy problem \eqref{abstracte} is said to be well-posed if for any $w_j\in\mathcal{D}(\mathfrak{L})$ $(j=0,\ldots,n-1)$ there exists a unique strong solution $w$ of equation \eqref{abstracte}, and $w_{j,m}\in\mathcal{D}(\mathfrak{L})$, $w_{j,m}\to0$ as $m\to+\infty$, imply $w\to0$ as $m\to+\infty$ in $X$, uniformly on compact time intervals.
\end{defn}

Let us now consider the following particular case of equation \eqref{abstracte}: 
\begin{equation}\label{abstractezero}
\prescript{C}{0}\partial^{\beta}_t w(t)=\mathfrak{L}w(t),\quad t>0;\quad w(0)=w_0\quad w^{(j)}(0)=0,\quad j=1,\ldots,n-1. 
\end{equation}
It can be shown that equation \eqref{abstracte} is well-posed if and only if the following integral equation 
\begin{equation}\label{abstracti}
w(t)=w_0+\prescript{RL}{0}{I}^{\beta}\big(\mathfrak{L}w\big)(t)    
\end{equation}
is well-posed in the sense of \cite[Def. 1.2]{pruss}. A solution operator of equation \eqref{abstractezero} will be defined in terms of the equivalent integral equation \eqref{abstracti}.  

\begin{defn}
A family $\{E_\beta(-t^{\beta}\mathfrak{L})\}_{t\geqslant0}\subset \mathcal{B}(X)$ is called a \textit{solution operator} of equation \eqref{abstractezero} if the following conditions hold:
\begin{enumerate}
    \item $E_\beta(-t^{\beta}\mathfrak{L})$ is strongly continuous for $t\geqslant0$ and $E_\beta(0)=I;$
    \item $E_\beta(-t^{\beta}\mathfrak{L})\mathcal{D}(\mathfrak{L})\subset\mathcal{D}(\mathfrak{L})$ and $\mathfrak{L}E_\beta(-t^{\beta}\mathfrak{L})w=E_\beta(-t^{\beta}\mathfrak{L})\mathfrak{L}w$ for any $w\in\mathcal{D}(\mathfrak{L})$, $t\geqslant0;$
    \item $E_\beta(-t^{\beta}\mathfrak{L})w$ is a solution of \eqref{abstracti} for any $w\in\mathcal{D}(\mathfrak{L})$, $t\geqslant0.$
\end{enumerate}
\end{defn}
Let us note that, by using the notions of \cite{pruss}, equation \eqref{abstractezero} is well-posed if and only if it has a solution operator. That is, if $E_{\beta}(-t^{\beta}\mathfrak{L})$ is a solution operator of equation \eqref{abstractezero} then the general equation \eqref{abstracte} is soluble (uniquely) and its solution is given by
\[
w(t)=\sum_{j=0}^{n-1}\prescript{RL}{0}I^{j}\big(E_\beta(-t^{\beta}\mathfrak{L})\big)(t)w_j,
\]
where $w_j\in\mathcal{D}(\mathfrak{L})$, $j=0,1,\ldots,n-1$, i.e. equation \eqref{abstracte}  is well-posed. Thus we can then study equation \eqref{abstractezero} instead of \eqref{abstracte}.

\subsection{Fourier analysis on a separable Hilbert space} \label{FA-hilbert}

We recall some preliminary results about the Fourier analysis which can be derived from a linear operator with discrete spectrum. This was recently developed in \cite{[15],fourier-1,[51]-1}. 

Let $\mfL:\text{Dom}(\mfL)\subset\mathcal{H}\to\mathcal{H}$ be a densely defined positive linear operator (with discrete spectrum $\{\lambda_\xi\}_{\xi\in I}$) in a separable Hilbert space $\mathcal{H}$. Recall that {\it the space of test functions for $\mfL$} is defined as  
\[
\mathcal{H}_{\mfL}^{\infty}:=\text{Dom}(\mfL^{\infty})=\bigcap_{k=1}^{+\infty}\text{Dom}(\mfL^k),
\]
where $\text{Dom}(\mfL^k)$ denotes the domain of the iterated operator $\mfL^k$, i.e. 
\[
\text{Dom}(\mfL^k)=\big\{g\in\mathcal{H}:\,\, \mfL^{i}g\in \text{Dom}(\mfL),\,\,i=0,1,\ldots,k-1\big\}.
\]
Therefore, the Fr\'echet topology on $\mathcal{H}_{\mfL}^{\infty}$ is given by the following family of semi-norms
\[
\|\phi\|_{\mathcal{H}_{\mfL}^{k}}=\max_{j\leqslant k}\|\mfL^{j}\phi\|_{\mathcal{H}},\quad k\in\mathbb{N}_0,\quad \phi\in \mathcal{H}_{\mfL}^{\infty}.
\]
Here we will always consider the system of eigenfunctions $\{e_\xi\}_{\xi\in I}$ of $\mfL$ to be an orthonormal basis in $\mathcal{H}.$ So, we have
\[
(e_\xi,e_\eta)_{\mh}=\delta_{\xi\eta}, \quad\text{where}\quad \text{$\delta_{\xi\eta}$ is the Kronecker delta,} 
\]
and $\|e_\xi\|_{\mathcal{H}}=1.$ Notice that $e_\xi\in \mh_{\mfL}^{\infty}$ for any $\xi\in I$. Thus the space $\mathcal{H}_{\mfL}^{\infty}$ is dense in $\mathcal{H}.$ 

Now we denote by $\mathcal{S}(I)$ the space of functions $\phi:I\to\mathbb{C}$ of rapid decay, i.e. such that for for any $N<+\infty$ there exists a constant $C_{\phi,N}$ such that 
\[
|\phi(\xi)|\leqslant C_{\phi,N}\langle \xi\rangle^{-N},\quad\text{for any $\xi\in I$,} 
\]
where $\langle \xi\rangle=(1+|\lambda_{\xi}|^2)^{\frac{1}{2N}}$ and $\lambda_\xi$ is the $\mfL$-eigenvalue associated with $e_\xi$ for each $\xi\in I$. The topology on $\mathcal{S}(I)$ is given by the following semi-norms:
\[
q_k(\phi)=\sup_{\xi\in I}\langle \xi\rangle^k |\phi(\xi)|, \quad k\in\mathbb{N}_0.
\]
The $\mfL$-Fourier transform is defined by  $\mathcal{F}_{\mfL}:\mathcal{H}^{\infty}_{\mfL}\to\mathcal{S}(I)$ such that 
\[
\big(\mathcal{F}_{\mfL}f\big)(\xi)=\widehat{f}(\xi)=(f,e_\xi).
\]
The Fourier transform is a bijective homeomorphism whose inverse is given by
\[
\mathcal{F}^{-1}_{\mfL}g=\sum_{\xi\in I}g(\xi)e_{\xi},\quad g\in \mathcal{S}(I).
\]
One can see that 
\[
f=\sum_{\xi\in I}\widehat{f}(\xi)e_\xi,\quad\text{for all}\quad f\in\mathcal{H}_{\mfL}^{\infty}. 
\]
Therefore, we get the Plancherel identity:
\[
\|f\|_{\mathcal{H}}^2=\sum_{\xi\in I}\big|\widehat{f}(\xi)\big|^2.
\]

Let us now recall a Sobolev space, which will be used frequently in the analysis of time-asymptotic behaviour of the solutions of the considered equations. So, for $\delta\in\mathbb{R}$, the \textit{Sobolev space $\mathcal{H}_{\mathfrak{L}}^{\delta}$} is defined by
\[
\mathcal{H}_{\mathfrak{L}}^{\delta}=\big\{f\in\mathcal{H}_{\mfL}^{\infty}:\,\, (I+\mathfrak{L})^{\delta/2}f\in \mathcal{H}\big\}
\]
endowed with the norm
\[
\|f\|_{\mathcal{H}_{\mathfrak{L}}^{\delta}}:=\|(I+\mathfrak{L})^{\delta/2}f\|_{\mathcal{H}}=\left(\sum_{\xi\in I}(1+\lambda_\xi)^{\delta}\big|\widehat{f}(\xi)\big|^2\right)^{1/2},
\]
where the last equality follows by the Plancherel identity. It is obvious that $\mathcal{H}_{\mathfrak{L}}^{\rho}\subset \mathcal{H}_{\mathfrak{L}}^{\delta}$ for any $\delta\leqslant\rho.$

Now we give several possibilities of settings where the above approach can be used. In general, self-adjoint operators with discrete spectrum are the most suitable ones. Regarding the the discrete spectra of self-adjoint operators we refer to \cite{book-self}. Hence Bessel or differential operators can be served as such examples. For example, problems related with the Sturm-Liouville, harmonic oscillator, anharmonic oscillator and Landau Hamiltonian meet our requests. We refer to \cite{fourier-1,RNT} for full details of the latter operators and problems. 

\subsection{Fourier analysis on graded Lie groups}\label{sub.int.graded}
In this subsection, we briefly recall some preliminary results on graded Lie groups and set up the related notation. More details can be found in \cite{FR16,FS82}.

A connected simply connected Lie group $G$ is called a \textit{graded Lie group} if its Lie algebra $\mathfrak{g}$ can be endowed with a vector space decomposition:
\[ 
 \mathfrak{g}= \bigoplus_{i=1}^{\infty}  \mathfrak{g}_{i}\,,
\]
such that all, but finitely many $\mathfrak{g}_{i}$'s, are  $\{0\}$ and $[\mathfrak{g}_{i},\mathfrak{g}_{j}]\subset \mathfrak{g}_{i+j}$.

A family of \textit{dilations} $\{D_r\}_{r>0}$  of a Lie algebra $\mathfrak{g}$ is a family of Lie algebra automorphisms $D_r : \mathfrak{g} \rightarrow \mathfrak{g}$ of the form $D_r=\textnormal{Exp}(A\, \textnormal{ln}r)$, where $A:\mathfrak{g} \rightarrow \mathfrak{g}$ is a diagonalisable linear operator with positive eigenvalues.

A Lie algebra that admits an automorphism of the form $D_r$ is nilpotent and as such it is an algebra of a connected simply connected nilpotent Lie group. Consequently, any dilation $D_r$ gives rise to a group automorphism $D_r: G \rightarrow G$ via the global diffeomorphism $\exp_{G}$ between $G$ and $\mathfrak{g}$.    

Let $(\pi,\mathcal{H}_\pi)$ be a unitary representation of $G$ on the separable Hilbert space $\mathcal{H}_{\pi}$. The vector $x \in \mathcal{H}_{\pi}$ is \textit{smooth} if the vector valued map $x \mapsto \pi(g)x$ from $G$ into $\mathcal{H}_{\pi}$ is smooth. In this case we write $x\in \mathcal{H}^{\infty}_{\pi}$. For a strongly continuous representation $\pi$, the limit 
\[
 d\pi(X)x:= \lim_{t \rightarrow 0} \frac{1}{t}\left(\pi(\exp_G(tX))x-x\right)\,,\quad x\in \mathcal{H}_{\pi}^{\infty}, X \in \mathfrak{g}\,,
 \]
 exists in the norm topology of $\mathcal{H}_\pi$, and the mapping $d\pi: \mathfrak{g}\rightarrow \textnormal{End}(\mathcal{H}_{\pi}^{\infty})$ is the \textit{infinitesimal representation} of $\mathfrak{g}$ on $\mathcal{H}_{\pi}^{\infty}$ associated to $\pi$. We can write $d\pi(X)=\pi(X)$.

 Recall that a consequence of the so-called \textit{Poincar\'{e}-Birkhoff-Witt} Theorem is that the universal enveloping algebra $\mathfrak{U}(\mathfrak{g})$ can be identified with the space of the left-invariant operators on $G$. The latter means in particular that if $T \in \mathfrak{U}(\mathfrak{g})$, then $T$ can be written as 
\[
T=\sum_{\alpha \in \mathcal{J}}c_{\alpha}X^{\alpha}\,,
\]
where $\mathcal{J} \subset \mathbb{N}^n$ is a finite set, and  $X^{\alpha}:=X_{1}^{\alpha_{1}}\cdots X_{n}^{\alpha_{n}}$, for $X_j \in \mathfrak{g}$ and $\alpha=(\alpha_{1},\cdots,\alpha_{n}) \in \mathbb{N}^n$. This identification allows extending the domain of the mapping $d\pi:=\pi$ to $\mathfrak{U}(\mathfrak{g})$. Additionally, we note that for $T \in \mathfrak{U}(\mathfrak{g})$, the field of operators $\{\pi(T) : \pi \in \widehat{G}\}$\footnote{By $\widehat{G}$ we denote the unitary dual of $G$; that is the set of all equivalence classes of irreducible, strongly
continuous and unitary representations of $G$.} is the \textit{symbol} in the sense of \cite[Def. 5.1.33]{FR16} associated to the operator $T$.

When $\mathfrak{g}$ is a graded Lie algebra, an important class of operators in  $\mathfrak{U}(\mathfrak{g})$ is that of \textit{Rockland operators}. Rockland operators, usually denoted by $\mathcal{R}$, are left-invariant differential operators, that are homogeneous (with respect to the dilations) of positive degree \cite[Def. 3.1.15]{FR16}, and satisfy the \textit{Rockland condition} \cite[Def. 4.1.1]{FR16}, hence they are hypoelliptic. \footnote{Let $T$ be a linear differential operator on a manifold $M$ with smooth coefficients. We say that $T$ is hypoelliptic, if for $u \in \mathcal{D}'(N)$ the condition $Tu \in C^{\infty}(N)$ implies that $u \in C^{\infty}(N)$ for any open $N \subset M$.}

Let us point out that the operator $\mathcal{R}$ on $\mathcal{D}(G)$ and $\pi(\mathcal{R})$ on $\mathcal{H}_{\pi}^{\infty}$ are densely defined on their domains. In the sequel, we shall keep the notation $\mathcal{R}$ and $\pi(\mathcal{R})$ for their self-adjoint extension on $L^2(G)$ and $\mathcal{H}_{\pi}^{\infty}$, respectively.

Regarding the representation of the operator $\pi(\mathcal{R})$ we have the following result: Let $\pi \in \widehat{G}\setminus \{1\}$, and let $\mathcal{R}$ be a positive Rockland operator on the graded group $G$.   In \cite{HJL85} the authors proved that the spectrum of the operator $\pi(\mathcal{R})$ is discrete and lies in $(0,\infty)$. The latter allows for an orthonormal basis for $\mathcal{H}_\pi$ which in turn gives rise to an infinite matrix representation of the form 
\begin{equation}\label{repr.pr}
\pi(\mathcal{R})=\begin{pmatrix}
\pi_{1}^{2} & 0 & \cdots & \cdots\\
0 & \pi_{2}^{2} & 0 & \cdots\\
\vdots & 0 & \ddots & \\
\vdots & \vdots & & \ddots
\end{pmatrix}\,.
\end{equation}

On our setting the \textit{group Fourier transform} is defined on $L^1(G,{\rm d}x)$ by 
\[
 \mathcal{F}_{G}f(\pi)\equiv \widehat{f}(\pi) \equiv \pi(f):= \int_{G}f(x)\pi(x)^{*}\,{\rm d}x\,,
\]
where $dx$ stands for the (bi-invariant) Haar measure on $G$. 

The Fourier transform satisfies the following property:  For $f \in \mathcal{S}(G) \cap L^1(G)$ and $X \in \mathfrak{g}$
\[
 \mathcal{F}_{G}(X f)(\pi)=\pi(X)\widehat{f}(\pi)\,.
\]
Consequently for the Rockland operator $\mathcal{R}$ on $G$ we get 
\[
 \mathcal{F}_{G}(\mathcal{R} f)(\pi)=\pi(\mathcal{R})\widehat{f}(\pi)\,,
\]
so that using \eqref{repr.pr} the operator has the following matrix representation 
\begin{equation}
\label{matrix.repr}
\mathcal{F}_{G}(\mathcal{R} f)(\pi)=\left\{ \pi_{i}^{2}\cdot \widehat{f}(\pi)_{i,j}\right\}_{i,j \in \mathbb{N}}\,.
\end{equation}
The orbit method, see \cite{CG90} and \cite{Kir04}, allows to describe $\widehat{G}$ as the subset of some Euclidean space. This allows to equip $\widehat{G}$ with a concrete measure, called in the literature the \textit{Plancherel measure} usually denoted by $\mu$. On the other hand for $f \in L^1(G) \cap L^2(G)$, and for $x \in G$, the operators $\pi(f)\pi(x)$, $\pi(x)\pi(f)$ and $\pi(f)$ are (under the same equivalent class $\pi$) trace class and Hilbert--Schmidt, respectively, and integrable against $\mu$. Under these considerations we have the isometry, known as the \textit{Plancherel formula}
\begin{equation}\label{plan.gr}
\int_{G} |f(x)|^2\,{\rm d}x=\int_{\widehat{G}}\textnormal{Tr}(\pi(f)\pi(f)^{*})\,{\rm d}\mu(\pi)=\int_{\widehat{{G}}}\|\pi(f)\|^{2}_{\textnormal{HS}(\mathcal{H}_{\pi})}\,{\rm d}\mu(\pi)\,,
\end{equation}
while any $f \in \mathcal{S}(G)$ may be recovered via the \textit{Fourier inversion formula} given by 
\begin{equation}
    \label{Four.inv.for}
f(x)=\int_{\widehat{G}}\textnormal{Tr}(\pi(x)\pi(f))\,{\rm d}\mu(\pi)=\int_{\widehat{G}}\textnormal{Tr}(\pi(f)\pi(x))\,{\rm d}\mu(\pi)\,.
\end{equation}
Let us point out that the orbit methods and its consequences hold true in the more general setting of a connected simply connected nilpotent Lie group.

\section{Heat and wave type equations on Hilbert spaces}\label{explicit}

In the first part of this section we apply the $\mfL$-Fourier transform (see subsection \ref{FA-hilbert}) to get the analytic solution of heat type equations by using a non-local integro-differential operator of Dzhrbashyan-Caputo type (in time). We complement our analysis with the study of the time-decay rate of the solutions. Here we show the asymptotic time estimates of solutions in the Sobolev space $\mathcal{H}_{\mfL}^{\delta}$ $(\delta\in\mathbb{R})$. In the next subsections we do a similar analysis for the wave type equation. Also, we consider multi-term heat and wave type equations.

\medskip As usual, in the sequel, we denote the spectrum of an operator $\mfL$ by $\sigma(\mfL).$

\medskip In this section we give the explicit solution and the time decay estimates for the $\mfL$-heat type equation, $\mfL$-wave type equation, as well as for their mutli-term variations.
\subsection{$\mfL$-heat type equation}

Let $\mfL:\text{Dom}(\mfL)\subset\mathcal{H}\to\mathcal{H}$ be a positive linear operator densely defined with discrete spectrum $\sigma(\mfL)=\{\gamma_\zeta:\zeta\in I\}$ ($I$ is a countable set of indices) on a separable Hilbert space $\mathcal{H}$, such that the system of eigenfunctions $\{e_\zeta\}_{\zeta \in I}$ of $\mfL$ is an orthonormal basis in $\mathcal{H}$. Also, we assume the operator to be closed if we work on a real Hilbert space. We consider the following equation
\begin{equation}\label{heat-FDE}
\prescript{C}{0}{\partial}^{\beta}w(t)+\mathfrak{L}w(t)=0,\qquad t>0,\quad 0<\beta\leqslant 1,
\end{equation}
under the initial condition
\[
w(t)|_{t=0}=w_0\in\mathcal{H}. 
\]

Before presenting our first result, let us introduce a useful notation to be involved in the explicit representation of the solution of our equations. For 
$\mathfrak{L}$ as above we define the Mittag-Leffler propagator 
\begin{equation}\label{Mittag-propagator}
E_{\nu,\rho}(-u\mathfrak{L}):=\sum_{k=0}^{+\infty}\frac{(-u\mathfrak{L})^k}{\Gamma(\nu k+\rho)},\quad u>0,\quad \nu,\rho\in\mathbb{R}.
\end{equation}
Notice that for any $z\in\mathbb{C}$, the function $E_{\nu,\rho}(z)$ coincides with the two parametric Mittag-Leffler function. For more details on different Mittag-leffler functions, see the book \cite{mittag}. 

Also, we will use frequently the following estimate from \cite[Theorem 1.6]{page 35}: 
\begin{equation}\label{uniform-estimate}
E_{\nu,\rho}(-s)\leqslant \frac{C}{1+s}\,,\quad s>0,\,\,\nu\in\mathbb{R},\,\, \rho<2,
\end{equation}
for some positive constant $C$. Moreover, by \cite[Theorem 4]{Mittag-bounded} we have the following
\begin{equation}
    \label{MLest2}
    E_{\nu,1}(-s):=E_{\nu}(-s)\leqslant \frac{1}{1+\Gamma(1+\nu)^{-1}s}\,,\quad s>0\,,\quad 0<\nu<1,
\end{equation}
with optimal constants.

\medskip The following is the first result about \eqref{heat-FDE}.
\begin{thm}\label{heat-thm}
Let $0<\beta\leqslant1$ and $\delta\in\mathbb{R}$. Let $\mathfrak{L}$ be a densely defined positive linear operator with discrete spectrum on a separable Hilbert space $\mathcal{H}$. 
\begin{enumerate}
    \item If $0\in\sigma(\mfL)$ and $w_0\in\mathcal{H}_{\mathfrak{L}}^{\delta}$ then there exists a unique continuous solution $w(t)\in \mathcal{H}^{\delta}_{\mfL}$ for any $t\in(0,+\infty)$ of equation \eqref{heat-FDE} represented by
\begin{equation}\label{solution-heat}
w(t)=\sum_{\zeta\in I}(w_0,e_\zeta)_{\mathcal{H}}E_{\beta}(-t^{\beta}\mathfrak{L})e_{\zeta}=\sum_{\zeta\in I}(w_0,e_\zeta)E_{\beta}(-\gamma_\zeta t^{\beta})e_{\zeta},
\end{equation}
such that 
\begin{equation}\label{heat-1}
\|w(t)\|_{\mathcal{H}^{\delta}_{\mathfrak{L}}}\lesssim \|w_0\|_{\mathcal{H}_{\mathfrak{L}}^{\delta}},\quad t\geqslant0.
\end{equation}
\item If $0\notin\sigma(\mfL)$ and $w_0\in\mathcal{H}_{\mathfrak{L}}^{\delta}$ then there exists a unique continuous solution $w(t)\in \mathcal{H}^{\delta}_{\mfL}$ for any $t\in(0,+\infty)$ of equation \eqref{heat-FDE} given by \eqref{solution-heat} 
such that 
\begin{equation}\label{heat-2}
\|w(t)\|_{\mathcal{H}^{\delta}_{\mathfrak{L}}}\lesssim (1+t^{\beta})^{-1}\|w_0\|_{\mathcal{H}_{\mathfrak{L}}^{\delta}},\quad t\geqslant0.
\end{equation}
\item If $w_0\in\mathcal{H}_{\mathfrak{L}}^{\delta-2}$ then there exists a unique continuous solution $w(t)\in \mathcal{H}^{\delta}_{\mfL}$ for any $t\in(0,+\infty)$ of equation \eqref{heat-FDE} given by \eqref{solution-heat} 
such that 
\begin{equation}\label{heat-3}
\|w(t)\|_{\mathcal{H}^{\delta}_{\mathfrak{L}}}\lesssim (1+t^{-\beta})\|w_0\|_{\mathcal{H}_{\mathfrak{L}}^{\delta-2}},\quad t>0.
\end{equation}
\end{enumerate}
\end{thm}
\begin{proof}

We first consider the set of eigenvalues $\{\gamma_\zeta\}_{\zeta\in I}$ and eigenfunctions $\{e_\zeta\}_{\zeta\in I}$ for the operator $\mathfrak{L}.$ Since the set $\{e_\zeta\}_{\zeta\in I}$ is a basis in $\mathcal{H}$ (orthonormal), we can  write the function $w$ as follows 
\[
w(t)=\sum_{\zeta\in I}w_\zeta(t)e_{\zeta}.
\]
By using the above expression in equation \eqref{heat-FDE} and applying the $\mfL$-Fourier transform we obtain  
\[
\prescript{C}{0}{\partial}^{\beta}w_\zeta(t)+\gamma_\zeta w_\zeta(t)=0,\qquad t>0\quad \text{and each}\quad \zeta\in I.
\]
Now we apply the Laplace transform in the time-variable:
\[
s^{\beta}\widehat{w_\zeta}(s)-s^{\beta-1}w_\zeta(0)+\gamma_\zeta \widehat{w_\zeta}(s)=0,\qquad s>0.
\]
Then 
\[
\widehat{w_{\zeta}}(s)=\frac{s^{\beta-1}}{s^{\beta}+\gamma_\zeta}w_\zeta(0).
\]
By using the inverse Laplace transform (see e.g. \cite[Theorem 2.3]{laplace-inverse} or \cite[Theorem 2.1]{new-mittag-add}), it follows that 
\begin{equation}\label{heat-estimate}
w_{\zeta}(t)=E_{\beta}(-\gamma_\zeta t^{\beta})w_\zeta(0).
\end{equation}
So
\begin{align*}
w(t)&=\sum_{\zeta\in I}E_{\beta}(-\gamma_\zeta t^{\beta})w_\zeta(0)e_{\zeta}=\sum_{k=0}^{+\infty}\frac{(-t^{\beta})^k}{\Gamma(\beta k+1)}\sum_{\zeta\in I}w_\zeta(0)\gamma_\zeta^{k}e_{\zeta} \\
&=\sum_{k=0}^{+\infty}\frac{(-t^{\beta})^k}{\Gamma(\beta k+1)}\sum_{\zeta\in I}w_\zeta(0)\mathfrak{L}^k e_{\zeta}=\sum_{\zeta\in I}w_\zeta(0)\sum_{k=0}^{+\infty}\frac{(-t^{\beta}\mathfrak{L})^k}{\Gamma(\beta k+1)} e_{\zeta} \\
&=\sum_{\zeta\in I}w_\zeta(0)E_{\beta}(-t^{\beta}\mathfrak{L})e_{\zeta}.
\end{align*}
By estimate \eqref{uniform-estimate} we obtain
\begin{align}\label{aster}
|w_{\zeta}(t)|=|E_{\beta}(-\gamma_\zeta t^{\beta})||w_\zeta(0)|\leqslant C\frac{|w_\zeta(0)|}{1+\gamma_\zeta t^{\beta}}\leqslant C|w_\zeta(0)|.
\end{align}
Thus
\begin{align*}
\|w(t)\|_{\mathcal{H}}^2 =\sum_{\zeta\in I}|w_\zeta(t)|^2&\leqslant C\|w_0\|^2_{\mathcal{H}}<+\infty,\quad t\in[0,+\infty),
\end{align*}
therefore the above sum is uniformly and absolutely convergent. Moreover, the Mittag-Leffler function involved in the series is an entire function \cite{mittag}, which provides the continuity of the solution.   

\medskip By the spectral calculus it is enough to prove the estimates of this theorem for the case $\delta=2.$ We have the following two cases:
\begin{enumerate}
\item { $\bf0\in\sigma(\mfL)$.} By equality \eqref{heat-estimate}, estimates  \eqref{aster} and \eqref{uniform-estimate} one can see  
\begin{align*}
\|w(t)\|_{\mathcal{H}_\mfL^2}^2&=\|(I+\mathfrak{L})w(t)\|_{\mathcal{H}}^2=\sum_{\zeta\in I}(1+\gamma_\zeta)^2|w_\zeta(t)|^2 \\
&\leqslant C\sum_{\zeta\in I}\frac{(1+\gamma_\zeta)^2}{(1+\gamma_\zeta t^{\beta})^2}|w_\zeta(0)|^2 \lesssim \|w_0\|_{\mathcal{H}^{2}_{\mfL}}^2,\quad t\geqslant0,
\end{align*}
which gives \eqref{heat-1}. 
\item {$\bf0\notin\sigma(\mfL)$.} We know that
\begin{align*}
\|w(t)\|_{\mathcal{H}_\mfL^2}^2&\leqslant C\sum_{\zeta\in I}\frac{(1+\gamma_\zeta)^2}{(1+\gamma_\zeta t^{\beta})^2}|w_\zeta(0)|^2 \\
&\leqslant C\frac{1}{(1+\gamma_\zeta^* t^{\beta})^2}\sum_{\zeta\in I}(1+\gamma_\zeta)^2|w_\zeta(0)|^2\lesssim (1+t^{\beta})^{-2} \|w_0\|_{\mathcal{H}^{2}_{\mfL}}^2,\quad t\geqslant0,
\end{align*}
where $\gamma_\zeta^*>0$ is the smallest eigenvalue of the sequence $\{\gamma_\zeta\}_{\zeta\in I}$ and thus we obtain inequality \eqref{heat-2}. 
\end{enumerate}
On the other hand, we also have 
\begin{align*}
    \frac{1+\gamma_\zeta}{1+\gamma_\zeta t^{\beta}}=\frac{1}{1+\gamma_\zeta t^{\beta}}+\frac{\gamma_\zeta}{1+\gamma_\zeta t^{\beta}}\leqslant 1+t^{-\beta}, 
\end{align*}
and it implies
\eqref{heat-3}. 

\end{proof}

\begin{rem}
Notice that the analysis of the $\mfL$-heat type equation \eqref{heat-FDE} is restricted to positive operators to guarantee the existence of the propagator in the considered space. Moreover, this is a necessary condition for the existence of the solution. Nevertheless, the $\mfL$-Fourier method can be applied to any linear operator densely defined. Also, if we consider a positive self adjoint operator in Theorem \ref{heat-thm}, this provides immediately the required condition on the basis formed with the eigenfunctions of the operator.    
\end{rem}

\subsection{$\mfL$-wave type equation}
Let $\mathfrak{L}$ be as before. We consider the following equation
\begin{equation}\label{wave-FDE}
\prescript{C}{0}{\partial}^{\beta}w(t)+\mathfrak{L}w(t)=0,\qquad t>0,\quad 1<\beta<2,
\end{equation}
under the initial conditions
\[
\partial_t^{(k)}w(t)|_{t=0}=w_k\in\mathcal{H},\quad k=0,1\,.
\]
We now give the main result for \eqref{wave-FDE}.
\begin{thm}
Let $1<\beta<2$ and $\delta\in\mathbb{R}$. Let $\mathfrak{L}$ be a positive linear operator densely defined with discrete spectrum on a separable Hilbert space $\mathcal{H}$.
\begin{enumerate}
    \item If $0\in\sigma(\mfL)$ and $w_0,w_1\in\mathcal{H}_{\mathfrak{L}}^{\delta}$ then there exists a unique continuous solution $w(t)\in \mathcal{H}^{\delta}_{\mfL}$ for any $t\in(0,+\infty)$ of equation \eqref{wave-FDE} given by
\begin{equation}\label{solution-wave}
w(t)=\sum_{\zeta\in I}\Bigg[(w_0,e_\zeta)_{\mathcal{H}}E_{\beta}(-t^{\beta}\mathfrak{L})+t(w_1,e_\zeta)_{\mathcal{H}}E_{\beta,2}(-t^{\beta}\mathfrak{L})\Bigg]e_{\zeta},
\end{equation}
such that 
\begin{equation}\label{wave-1}
\|w(t)\|_{\mathcal{H}^{\delta}_{\mathfrak{L}}}\lesssim \|w_0\|_{\mathcal{H}_{\mathfrak{L}}^{\delta}}+t\|w_1\|_{\mathcal{H}_{\mathfrak{L}}^{\delta}},\quad t\geqslant0.
\end{equation}
\item If $0\notin\sigma(\mfL)$ and $w_0,w_1\in\mathcal{H}_{\mathfrak{L}}^{\delta}$ then there exists a unique continuous solution $w(t)\in \mathcal{H}^{\delta}_{\mfL}$ for any $t\in(0,+\infty)$ of equation \eqref{wave-FDE} represented by \eqref{solution-wave} such that 
\begin{equation}\label{wave-2}
\|w(t)\|_{\mathcal{H}^{\delta}_{\mathfrak{L}}}\lesssim (1+t^{\beta})^{-1}\|w_0\|_{\mathcal{H}_{\mathfrak{L}}^{\delta}}+t(1+t^{\beta})^{-1}\|w_1\|_{\mathcal{H}_{\mathfrak{L}}^{\delta}},\quad t\geqslant0.
\end{equation}
\item If $(w_0,w_1)\in\bigg(\mathcal{H}_{\mathfrak{L}}^{\delta},\mathcal{H}_{\mathfrak{L}}^{\delta-\frac{2}{\beta}}\bigg)$ then there exists a unique continuous solution $w(t)\in \mathcal{H}^{\delta+2}_{\mfL}$ for any $t\in(0,+\infty)$ of equation \eqref{wave-FDE} represented by \eqref{solution-wave} such that 
\begin{equation}\label{wave-3}
\|w(t)\|_{\mathcal{H}^{\delta}_{\mathfrak{L}}}\lesssim \|w_0\|_{\mathcal{H}_{\mathfrak{L}}^{\delta}}+(1+t)\|w_1\|_{\mathcal{H}_{\mathfrak{L}}^{\delta-\frac{2}{\beta}}},\quad t\geqslant0.
\end{equation}
\item If $(w_0,w_1)\in\big(\mathcal{H}_{\mathfrak{L}}^{\delta},\mathcal{H}_{\mathfrak{L}}^{\delta-2}\big)$ then there exists a unique continuous solution $w(t)\in \mathcal{H}^{\delta}_{\mfL}$ for any $t\in(0,+\infty)$ of equation \eqref{wave-FDE} represented by \eqref{solution-wave} such that 
\begin{equation}\label{wave-4}
\|w(t)\|_{\mathcal{H}^{\delta}_{\mathfrak{L}}}\lesssim \|w_0\|_{\mathcal{H}_{\mathfrak{L}}^{\delta}}+t(1+t^{-\beta})\|w_1\|_{\mathcal{H}_{\mathfrak{L}}^{\delta-2}},\quad t>0.
\end{equation}
\end{enumerate}
\end{thm}

\begin{proof}
By the application of the $\mfL$-Fourier transform to equation \eqref{wave-FDE} we get 
\[
\prescript{C}{0}{\partial}^{\beta}w_\zeta(t)+\gamma_\zeta w_\zeta(t)=0,\qquad t>0,\quad\text{for each}\quad \zeta\in I.
\]
By the Laplace transform one has
\[
s^{\beta}\widehat{w_\zeta}(s)-s^{\beta-1}w_\zeta(0)-s^{\beta-2}\partial_t^{(1)}w_\zeta(t)\big|_{_{_{_{t=0}}}}+\gamma_\zeta \widehat{w_\zeta}(s)=0,\qquad s>0,
\]
and 
\[
\widehat{w_{\zeta}}(s)=\frac{s^{\beta-1}}{s^{\beta}+\gamma_\zeta}w_\zeta(0)+\frac{s^{\beta-2}}{s^{\beta}+\gamma_\zeta}\partial_t^{(1)}w_\zeta(t)\big|_{_{_{t=0}}}.
\]
By the inverse Laplace transform (see e.g. \cite[Theorem 2.3]{laplace-inverse} or \cite[Theorem 2.1]{new-mittag-add}), we arrive at 
\begin{align*}
w_{\zeta}(t)&=E_{\beta}(-\gamma_\zeta t^{\beta})w_\zeta(0)+tE_{\beta,2}(-\gamma_\zeta t^{\beta})\partial_t^{(1)}w_\zeta(t)\big|_{_{_{_{t=0}}}} \\
&=(w_0,e_\zeta)E_{\beta}(-\gamma_\zeta t^{\beta})+t(w_1,e_\zeta)E_{\beta,2}(-\gamma_\zeta t^{\beta}).
\end{align*}
Hence
\begin{align*}
w(t)&=\sum_{\zeta\in I}w_\zeta(t)e_{\zeta}=\sum_{\zeta\in I}\left[(w_0,e_\zeta)E_{\beta}(-\gamma_\zeta t^{\beta})+t(w_1,e_\zeta)E_{\beta,2}(-\gamma_\zeta t^{\beta})\right]e_{\zeta} \\
&=\sum_{\zeta\in I}\left[(w_0,e_\zeta)\left(\sum_{k=0}^{+\infty}\frac{(-t^{\beta})^k}{\Gamma(\beta k+1)}\gamma_\zeta^k\right) +t(w_1,e_\zeta)\left(\sum_{k=0}^{+\infty}\frac{(-t^{\beta})^k}{\Gamma(\beta k+2)}\gamma_\zeta^k\right)\right]e_{\zeta},
\end{align*}
which gives the desired representation since $\gamma_\zeta^k e_\zeta=\mathfrak{L}^k e_\zeta.$ 

Notice now that by \eqref{uniform-estimate} we get
\begin{align}\label{wave-estimate}
|w_{\zeta}(t)| &\leqslant |E_{\beta}(-\gamma_\zeta t^{\beta})||(w_0,e_\zeta)|+t|E_{\beta,2}(-\gamma_\zeta t^{\beta})||(w_1,e_\zeta)| \nonumber \\
&\leqslant \frac{C}{1+\gamma_\zeta t^{\beta}}\big(|(w_0,e_\zeta)|+t|(w_1,e_\zeta)|\big).
\end{align}
Again, by the spectral calculus it is enough to prove the estimates of this theorem for the case $\delta=2.$ By \eqref{wave-estimate} we have 
\begin{align*}
\|w(t)\|_{\mathcal{H}_\mfL^2}^2&=\sum_{\zeta\in I}(1+\gamma_\zeta)^2|w_\zeta(t)|^2 \\
&\leqslant C\sum_{\zeta\in I}\frac{(1+\gamma_\zeta)^2}{(1+\gamma_\zeta t^{\beta})^2}\big(|(w_0,e_\zeta)|+t|(w_1,e_\zeta)|\big)^2 \\
&\leqslant C\left(\sum_{\zeta\in I}\frac{(1+\gamma_\zeta)^2}{(1+\gamma_\zeta t^{\beta})^2}|(w_0,e_\zeta)|^2+t^2\sum_{\zeta\in I}\frac{(1+\gamma_\zeta)^2}{(1+\gamma_\zeta t^{\beta})^2}|(w_1,e_\zeta)|^2\right) \\
&:=J_1(t,w_0)+J_2(t,w_1).
\end{align*}
The analysis of $J_1(t,w_0)$ coincides with the one in the proof of Theorem \ref{heat-thm} even if the range of $\beta$ is different. 

We begin with the case ${\bf 0\in\sigma(\mfL).}$ We know that $J_1(t,w_1)\lesssim \|w_0\|_{\mathcal{H}_{\mathfrak{L}}^{2}}^2$ and clearly $J_2(t,w_1)\leqslant t^2\|w_1\|_{\mathcal{H}_{\mfL}^2}^2$. So, we arrive at \eqref{wave-1}.

The inequality \eqref{wave-2} follows by the same analysis of \eqref{heat-2}.

\medskip We will next prove \eqref{wave-3} and \eqref{wave-4}. We rewrite $J_2$ as
\begin{align*}
J_2(t,w_1)=C\sum_{\zeta\in I}\left(\frac{t}{1+\gamma_\zeta t^{\beta}}+\frac{t\gamma_\zeta}{1+\gamma_\zeta t^{\beta}}\right)^2 |(w_1,e_\zeta)|^2.
\end{align*}
We estimate the first term by $t^2$. While in the second term we have two options. First
\begin{align*}
    \frac{t\gamma_\zeta}{1+\gamma_\zeta t^{\beta}}\leqslant t^{1-\beta}\Longrightarrow J_2(t,w_1)\lesssim (t+t^{1-\beta})^2\|w_1\|_{\mathcal{H}}^2,
\end{align*}
and second 
\begin{align*}
    \frac{t\gamma_\zeta}{1+\gamma_\zeta t^{\beta}}\leqslant \gamma_\zeta \sup_{t>0}\frac{t}{1+\gamma_\zeta t^{\beta}}=\gamma_\zeta\frac{t}{1+\gamma_\zeta t^{\beta}}\Bigg|_{_{_{t=1/(\gamma_\zeta (\beta-1))^{1/{\beta}}}}}\leqslant C_{\beta}\gamma_\zeta^{1-\frac{1}{\beta}},
\end{align*}
which implies that
\begin{align*}
J_2(t,w_1)&\leqslant C\left(t^2\sum_{\zeta\in I}|(w_1,e_\zeta)|^2+\sum_{\zeta\in I}\left(\frac{t\gamma_\zeta}{1+\gamma_\zeta t^{\beta}}\right)^2 |(w_1,e_\zeta)|^2\right) \\
&\lesssim t^2\|w_1\|_{\mathcal{H}}^2+\|w_1\|_{\mathcal{H}_{\mfL}^{\frac{2(\beta-1)}{\beta}}}^2 \lesssim (1+t^2)\|w_1\|_{\mathcal{H}_{\mfL}^{2-\frac{2}{\beta}}}^2,
\end{align*}
since $\mathcal{H}_{\mathfrak{L}}^{\rho}\subset \mathcal{H}_{\mathfrak{L}}^{\delta}$ for any $\delta\leqslant\rho.$ Therefore, we get
\begin{align*}
J_2(t,w_1)&\lesssim \left\{
\begin{array}{rccl}
&t^2(1+t^{-\beta})^2\|w_1\|_{\mathcal{H}}^2,  \\
&(1+t)^2\|w_1\|_{\mathcal{H}_{\mfL}^{2-\frac{2}{\beta}}}^2,
\end{array}
\right.
\end{align*}
along with $J_1(t,w_0)\lesssim (1+t^{-\beta})^2 \|w_0\|_{\mathcal{H}_{\mathfrak{L}}^{2}}^2,$ which proves inequalities \eqref{wave-3} and \eqref{wave-4}. The proof is complete.
\end{proof}

\subsection{Multi-term $\mfL$-heat type equations}\label{multi-section-heat}

Now we focus on the case of multi-term heat type equations. We study the following equation:  
\begin{equation}\label{multi-heat}
\left\{ \begin{split}
\prescript{C}{}\partial_{t}^{\beta}w(t)+\sigma_1\prescript{C}{}\partial_{t}^{\beta_1}w(t)+\cdots+\sigma_m\prescript{C}{}\partial_{t}^{\beta_m}w(t)+\mfL w(t)&=0,   \\
w(t)\big|_{t=0}\big.&=w_0\in\mathcal{H},
\end{split}
\right.
\end{equation}
where $0<t\leqslant T<+\infty$ and $\mathfrak{L}$ is a positive linear operator densely defined with discrete spectrum on a separable Hilbert space $\mathcal{H}$, $\sigma_i\geqslant0$ $(i=1,\ldots,m)$ and $1\geqslant\beta>\beta_1>\cdots>\beta_{m}>0$. 

The solution of equation \eqref{multi-heat} is related with the multivariate Mittag-Leffler function \cite{ML-defined,multi-ML}, see also \cite{new-mittag-add}. So, let us recall the definition of this important special function, which is absolutely and locally uniformly convergent for the given parameters.  
\begin{defn}\label{multivariate-def}
Let $\beta_{i},\lambda\in\mathbb{R}$ $(i=1,\ldots,m)$ with $\beta_{i}>0$. The multivariate Mittag-Leffler function is defined as (\cite{ML-defined})
\begin{equation}\label{multivariateML}
E_{(\beta_1,\ldots,\beta_m),\lambda}(z_1,\ldots,z_m)=\sum_{k_1=0}^{+\infty}\cdots\sum_{k_m=0}^{+\infty}\frac{(k_1+\cdots+k_m)!}{\Gamma(\beta_1 k_1+\cdots+\beta_m k_m+\lambda)}\frac{z_1^{k_1}}{k_{1}!}\cdots\frac{z_m^{k_m}}{k_{m}!},
\end{equation}
for any complex numbers $z_1,\ldots,z_m\in\mathbb{C}$.
\end{defn}
We now give the main result for \eqref{multi-heat}.
\begin{thm}\label{multi-thm}
Let $\mathfrak{L}$ be a positive linear operator densely defined with discrete spectrum on a separable Hilbert space $\mathcal{H}$. Let $\delta\in\mathbb{R}$ and $1\geqslant\beta>\beta_1>\cdots>\beta_{m}>0$. 
\begin{enumerate}
    \item If $0\in\sigma(\mfL)$ and $w_0\in \mathcal{H}^{\delta}_{\mfL}$ then there exists a unique continuous solution $w(t)\in \mathcal{H}^{\delta}_{\mathfrak{L}}$ for any $t\in(0,T]$ for the equation \eqref{multi-heat} given by
\begin{align}\label{solution-heat-multi}
w(t)&=\sum_{k=0}^{m}\sigma_k t^{\beta-\beta_k}\sum_{\zeta\in I}(w_0,e_\zeta)_{\mathcal{H}}E_{(\beta-\beta_1,\ldots,\beta-\beta_m,\beta),\beta-\beta_k+1}(-\sigma_1 t^{\beta-\beta_1},\ldots \\
&\hspace{9cm}\ldots,-\sigma_m t^{\beta-\beta_m},-t^{\beta}\mfL)e_{\zeta} \nonumber\\
&:=\sum_{k=0}^{m}\sigma_k t^{\beta-\beta_k}\sum_{\zeta\in I}(w_0,e_\zeta)_{\mathcal{H}}E_{(\beta-\beta_1,\ldots,\beta-\beta_m,\beta),\beta-\beta_k+1}(-\sigma_1 t^{\beta-\beta_1},\ldots \nonumber \\
&\hspace{9cm}\ldots,-\sigma_m t^{\beta-\beta_m},-\gamma_\zeta t^{\beta}), \nonumber
\end{align}
where $\sigma_0=1$ and we have 
\begin{align*}
\|w(t)\|_{\mathcal{H}_\mfL^{\delta}}\leqslant C_{T,\vec{\beta},\vec{\sigma}} \left(\sum_{k=0}^{m}\sigma_k t^{\beta-\beta_k}\right)\|w_0\|_{\mathcal{H}_\mfL^{\delta}}, \quad 0<t\leqslant T, 
\end{align*}
for some constant which depends on $T$, $\vec{\beta}=(\beta,\ldots,\beta_m)$ and $\vec{\sigma}=(\sigma_0,\ldots,\sigma_m).$

\item If $0\notin\sigma(\mfL)$ and $w_0\in \mathcal{H}^{\delta}_{\mfL}$ then there exists a unique continuous solution $w(t)\in \mathcal{H}^{\delta}_{\mathfrak{L}}$ for any $t\in(0,T]$ for the equation \eqref{multi-heat} given by \eqref{solution-heat-multi} and we have 
\begin{align*}
\|w(t)\|_{\mathcal{H}_\mfL^{\delta}}\leqslant C_{T,\vec{\beta},\vec{\sigma}} \left(\sum_{k=0}^{m}\sigma_k t^{\beta-\beta_k}\right)(1+t^{\beta})^{-1}\|w_0\|_{\mathcal{H}_\mfL^{\delta}}, \quad 0<t\leqslant T. 
\end{align*}

\item If $w_0\in \mathcal{H}^{\delta-2}_{\mfL}$ then there exists a unique continuous solution $w(t)\in \mathcal{H}^{\delta}_{\mathfrak{L}}$ for any $t\in(0,T]$ for the equation \eqref{multi-heat} given by \eqref{solution-heat-multi} and we have 
\begin{align*}
\|w(t)\|_{\mathcal{H}_\mfL^{\delta}}\leqslant C_{T,\vec{\beta},\vec{\sigma}}(1+t^{-\beta})\left(\sum_{k=0}^{m}\sigma_k t^{\beta-\beta_k}\right)\|w_0\|_{\mathcal{H}_\mfL^{\delta-2}},\quad 0<t\leqslant T.
\end{align*}
\end{enumerate}
\end{thm}
\begin{proof}
Notice first that it is sufficient to prove the result for $\delta=2$. By applying the $\mfL$-Fourier transform to equation \eqref{multi-heat} we obtain
\[
\prescript{C}{}\partial_{t}^{\beta}w_\zeta(t)+\sigma_1\prescript{C}{}\partial_{t}^{\beta_1}w_\zeta(t)+\cdots+\sigma_m\prescript{C}{}\partial_{t}^{\beta_m}w_\zeta(t)+\gamma_\zeta w_\zeta(t)=0,\quad t>0,\quad\text{for each}\quad \zeta\in I.
\]
By using the Laplace transform in the time-variable one gets
\[
\begin{split}
s^{\beta}\widehat{w_\zeta}(s)-s^{\beta-1}w_\zeta(0)+\sigma_1 s^{\beta_1}\widehat{w_\zeta}(s)-\sigma_1 s^{\beta_1-1}w_\zeta(0)+&\cdots \\
&\hspace{-7cm}\cdots+\sigma_m s^{\beta_m}\widehat{w_\zeta}(s)-\sigma_m s^{\beta_m-1}w_\zeta(0)+\gamma_\zeta\widehat{w_\zeta}(s)=0, \quad s>0.
\end{split}
\]
Thus
\[
\widehat{w_\zeta}(s)=\frac{s^{\beta-1}+\sigma_1 s^{\beta_1-1}+\cdots+\sigma_m s^{\beta_m-1}}{s^{\beta}+\sigma_1 s^{\beta_1}+\cdots+\sigma_m s^{\beta_m}+\gamma_\zeta} w_\zeta(0), \quad s>0. 
\]
By applying the inverse Laplace transform and \cite[Theorem 2.3]{laplace-inverse} we arrive at
\begin{align*}
w_\zeta(t)=\sum_{k=0}^{m}\sigma_k t^{\beta-\beta_k}&E_{(\beta-\beta_1,\ldots,\beta-\beta_m,\beta),\beta-\beta_k+1}(-\sigma_1 t^{\beta-\beta_1},\ldots \\
&\hspace{3cm}\ldots,-\sigma_m t^{\beta-\beta_m},-\gamma_\zeta t^{\beta})w_\zeta(0),
\end{align*}
where $\sigma_0=1$ and $\beta_0=\beta.$ Hence by \cite[Lemma 3.3]{karel} (see also \cite[Lemma 3.2]{multi-estimate} or \cite[Lemma 3]{otra-estimacion}) we get
\begin{align}
|w_\zeta(t)|&\leqslant \sum_{k=0}^{m}\sigma_k t^{\beta-\beta_k}|E_{(\beta-\beta_1,\ldots,\beta-\beta_m,\beta),\beta-\beta_k+1}(-\sigma_1 t^{\beta-\beta_1},\ldots \nonumber\\
&\hspace{5cm}\ldots,-\sigma_m t^{\beta-\beta_m},-\gamma_\zeta t^{\beta})||w_\zeta(0)| \nonumber\\
&\leqslant C_{T,\sigma_1,\ldots,\sigma_m,\beta,\ldots,\beta_m}\frac{|w_\zeta(0)|}{1+\gamma_\zeta t^{\beta}}\sum_{k=0}^{m}\sigma_k t^{\beta-\beta_k},\quad 0<t\leqslant T.\label{multi-estimate}
\end{align}
Thus
\begin{align*}
\|w(t)\|_{\mathcal{H}_\mfL^2}^2&=\sum_{\zeta\in I}(1+\gamma_\zeta)^2|w_\zeta(t)|^2 \\
&\leqslant C_{T,\vec{\beta},\vec{\sigma}}^2 \left(\sum_{k=0}^{m}\sigma_k t^{\beta-\beta_k}\right)^2\sum_{\zeta\in I}\frac{(1+\gamma_\zeta)^2}{(1+\gamma_\zeta t^{\beta})^2}|w_\zeta(0)|^2 \\
&\leqslant 
\left\{
\begin{array}{rccl}
&\displaystyle C_{T,\vec{\beta},\vec{\sigma}}^2 \left(\sum_{k=0}^{m}\sigma_k t^{\beta-\beta_k}\right)^2\|w_0\|_{\mathcal{H}_\mfL^2}^2, &0<t\leqslant T,\quad 0\in\sigma(\mfL), \\
&\displaystyle C_{T,\vec{\beta},\vec{\sigma}}^2 \left(\sum_{k=0}^{m}\sigma_k t^{\beta-\beta_k}\right)^2 (1+t^{\beta})^{-2}\|w_0\|_{\mathcal{H}_\mfL^2}^2, &0<t\leqslant T,\quad 0\notin\sigma(\mfL), \\
&\displaystyle C_{T,\vec{\beta},\vec{\sigma}}^2 (1+t^{-\beta})^2\left(\sum_{k=0}^{m}\sigma_k t^{\beta-\beta_k}\right)^2\|w_0\|_{\mathcal{H}}^2, &0<t\leqslant T.
\end{array}
\right.
\end{align*}
The proof is complete.
\end{proof}

\subsection{Multi-term $\mfL$-wave type equations}\label{multi-section-wave}

We finish this section by studying the case of multi-term wave type equations. We consider the following equation:  
\begin{equation}\label{multi-wave}
\left\{ \begin{split}
\prescript{C}{}\partial_{t}^{\beta}w(t)+\sigma_1\prescript{C}{}\partial_{t}^{\beta_1}w(t)+\cdots+\sigma_m\prescript{C}{}\partial_{t}^{\beta_m}w(t)+\mfL w(t)&=0,  \\
\partial_t^{(k)}w(t)\big|_{t=0}\big.&=w_k\in\mathcal{H},\,\, k=0,1,
\end{split}
\right.
\end{equation}
where $0<t\leqslant T<+\infty$ and $\mathfrak{L}$ is a positive linear operator densely defined with discrete spectrum on a separable Hilbert space $\mathcal{H}$, $\sigma_i\geqslant0$ $(i=1,\ldots,m)$, $2>\beta>1$, and $\beta>\beta_1>\cdots>\beta_{m}>0$. 

\medskip The following is the main result for \eqref{multi-wave}.
\begin{thm}\label{multi-thm-wave}
Let $\mathfrak{L}$ be a positive linear operator densely defined with discrete spectrum on a separable Hilbert space $\mathcal{H}$. Let $\delta\in\mathbb{R}$, $2>\beta>1$ and $\beta>\beta_1>\cdots>\beta_{m}>0$. 
\begin{enumerate}
    \item If $0\in\sigma(\mfL)$ and $w_0,w_1\in \mathcal{H}^{\delta}_{\mfL}$ then there exists a unique continuous solution $w(t)\in \mathcal{H}^{\delta}_{\mathfrak{L}}$ for any $t\in(0,T]$ for the equation \eqref{multi-wave} given by
\begin{align}\label{solution-wave-multi}
w(t)&=\sum_{k=0}^{m}\sigma_k t^{\beta-\beta_k}\sum_{\zeta\in I}(w_0,e_\zeta)_{\mathcal{H}}E_{(\beta-\beta_1,\ldots,\beta-\beta_m,\beta),\beta-\beta_k+1}(-\sigma_1 t^{\beta-\beta_1},\ldots \\
&\hspace{9cm}\ldots,-\sigma_m t^{\beta-\beta_m},-\gamma_\zeta t^{\beta}), \nonumber \\
&+\sum_{k=0}^{m}\sigma_k t^{\beta-\beta_k+1}\sum_{\zeta\in I}(w_1,e_\zeta)_{\mathcal{H}}E_{(\beta-\beta_1,\ldots,\beta-\beta_m,\beta),\beta-\beta_k+2}(-\sigma_1 t^{\beta-\beta_1},\ldots \nonumber \\
&\hspace{9cm}\ldots,-\sigma_m t^{\beta-\beta_m},-\gamma_\zeta t^{\beta}), \nonumber
\end{align}
where $\sigma_0=1, \beta_0=\beta$ and we have 
\begin{align*}
\|w(t)\|_{\mathcal{H}_\mfL^{\delta}}\leqslant C_{T,\vec{\beta},\vec{\sigma}} \left(\sum_{k=0}^{m}\sigma_k t^{\beta-\beta_k}\right)\big(\|w_0\|_{\mathcal{H}_\mfL^{\delta}}+t\|w_1\|_{\mathcal{H}_\mfL^{\delta}}\big), \quad 0<t\leqslant T, 
\end{align*}
for some constant which depends on $T$, $\vec{\beta}=(\beta,\ldots,\beta_m)$ and $\vec{\sigma}=(\sigma_0,\ldots,\sigma_m).$

\item If $0\notin\sigma(\mfL)$ and $w_0,w_1\in \mathcal{H}^{\delta}_{\mfL}$ then there exists a unique continuous solution $w(t)\in \mathcal{H}^{\delta}_{\mathfrak{L}}$ for any $t\in(0,T]$ for the equation \eqref{multi-wave} given by \eqref{solution-wave-multi} and we have 
\begin{align*}
\|w(t)\|_{\mathcal{H}_\mfL^{\delta}}\leqslant C_{T,\vec{\beta},\vec{\sigma}} \left(\sum_{k=0}^{m}\sigma_k t^{\beta-\beta_k}\right)(1+t^{\beta})^{-1}\big(\|w_0\|_{\mathcal{H}_\mfL^{\delta}}+t\|w_1\|_{\mathcal{H}_\mfL^{\delta}}\big), \quad 0<t\leqslant T. 
\end{align*}
\item If $w_0\in \mathcal{H}^{\delta}_{\mfL}$ and $w_1\in \mathcal{H}^{\delta-\frac{2}{\beta}}_{\mfL}$ then there exists a unique continuous solution $w(t)\in \mathcal{H}^{\delta}_{\mathfrak{L}}$ for any $t\in(0,T]$ for the equation \eqref{multi-wave} given by \eqref{solution-wave-multi} and we have 
\begin{align*}
\|w(t)\|_{\mathcal{H}_\mfL^{\delta}}\leqslant C_{T,\vec{\beta},\vec{\sigma}}\left(\sum_{k=0}^{m}\sigma_k t^{\beta-\beta_k}\right)\bigg(\|w_0\|_{\mathcal{H}_\mfL^{\delta}}+(1+t)\|w_1\|_{\mathcal{H}_{\mfL}^{\delta-\frac{2}{\beta}}}\bigg),\quad 0<t\leqslant T.
\end{align*}
\item If $w_0,w_1\in \mathcal{H}^{\delta-2}_{\mfL}$ then there exists a unique continuous solution $w(t)\in \mathcal{H}^{\delta}_{\mathfrak{L}}$ for any $t\in(0,T]$ for the equation \eqref{multi-wave} given by \eqref{solution-wave-multi} and we have 
\begin{align*}
\|w(t)\|_{\mathcal{H}_\mfL^{\delta}}\leqslant C_{T,\vec{\beta},\vec{\sigma}}(1+t^{-\beta})\left(\sum_{k=0}^{m}\sigma_k t^{\beta-\beta_k}\right)\big(\|w_0\|_{\mathcal{H}_\mfL^{\delta-2}}+t\|w_1\|_{\mathcal{H}_\mfL^{\delta-2}}\big),\quad 0<t\leqslant T.
\end{align*}
\end{enumerate}
\end{thm}
\begin{proof}
By applying the $\mfL$-Fourier transform to equation \eqref{multi-wave} we obtain
\[
\prescript{C}{}\partial_{t}^{\beta}w_\zeta(t)+\sigma_1\prescript{C}{}\partial_{t}^{\beta_1}w_\zeta(t)+\cdots+\sigma_m\prescript{C}{}\partial_{t}^{\beta_m}w_\zeta(t)+\gamma_\zeta w_\zeta(t)=0,\quad t>0,\quad\text{for each}\quad \zeta\in I.
\]
By using the Laplace transform in the time-variable one and assuming (without lossing the generality) $\beta_i>1$ for $i=1,\ldots,m$ we get 
\[
\begin{split}
s^{\beta}\widehat{w_\zeta}(s)&-s^{\beta-1}(w_0,e_\zeta)-s^{\beta-2}(w_1,e_\zeta) \\
&+\sigma_1 s^{\beta_1}\widehat{w_\zeta}(s)-\sigma_1 s^{\beta_1-1}(w_0,e_\zeta)-\sigma_1 s^{\beta_1-2}(w_1,e_\zeta)+\cdots \\
&\cdots+\sigma_m s^{\beta_m}\widehat{w_\zeta}(s)-\sigma_m s^{\beta_m-1}(w_0,e_\zeta)-\sigma_m s^{\beta_m-2}(w_1,e_\zeta)+\gamma_\zeta\widehat{w_\zeta}(s)=0, \,\, s>0.
\end{split}
\]
Thus
\begin{align*}
\widehat{w_\zeta}(s)=&\frac{s^{\beta-1}+\sigma_1 s^{\beta_1-1}+\cdots+\sigma_m s^{\beta_m-1}}{s^{\beta}+\sigma_1 s^{\beta_1}+\cdots+\sigma_m s^{\beta_m}+\gamma_\zeta}(w_0,e_\zeta) \\
&+\frac{s^{\beta-2}+\sigma_1 s^{\beta_1-2}+\cdots+\sigma_m s^{\beta_m-2}}{s^{\beta}+\sigma_1 s^{\beta_1}+\cdots+\sigma_m s^{\beta_m}+\gamma_\zeta}(w_1,e_\zeta)
\quad s>0. 
\end{align*}
By applying the inverse Laplace transform and \cite[Theorem 2.3]{laplace-inverse} we arrive at
\begin{align*}
w_\zeta(t)&=\sum_{k=0}^{m}\sigma_k t^{\beta-\beta_k}E_{(\beta-\beta_1,\ldots,\beta-\beta_m,\beta),\beta-\beta_k+1}(-\sigma_1 t^{\beta-\beta_1},\ldots \\
&\hspace{5cm}\ldots,-\sigma_m t^{\beta-\beta_m},-\gamma_\zeta t^{\beta})(w_0,e_\zeta) \\
&+\sum_{k=0}^{m}\sigma_k t^{\beta-\beta_k+1}E_{(\beta-\beta_1,\ldots,\beta-\beta_m,\beta),\beta-\beta_k+2}(-\sigma_1 t^{\beta-\beta_1},\ldots \\
&\hspace{5cm}\ldots,-\sigma_m t^{\beta-\beta_m},-\gamma_\zeta t^{\beta})(w_1,e_\zeta),
\end{align*}
where $\sigma_0=1$ and $\beta_0=\beta.$ Hence by \cite[Lemma 3.3]{karel} (\cite[Lemma 3.2]{multi-estimate} or \cite[Lemma 3]{otra-estimacion}) we get
\begin{align}
|w_\zeta(t)|&\leqslant \sum_{k=0}^{m}\sigma_k t^{\beta-\beta_k}|E_{(\beta-\beta_1,\ldots,\beta-\beta_m,\beta),\beta-\beta_k+1}(-\sigma_1 t^{\beta-\beta_1},\ldots \nonumber\\
&\hspace{5cm}\ldots,-\sigma_m t^{\beta-\beta_m},-\gamma_\zeta t^{\beta})||(w_0,e_\zeta)| \nonumber\\
&+\sum_{k=0}^{m}\sigma_k t^{\beta-\beta_k+1}|E_{(\beta-\beta_1,\ldots,\beta-\beta_m,\beta),\beta-\beta_k+2}(-\sigma_1 t^{\beta-\beta_1},\ldots \nonumber\\
&\hspace{5cm}\ldots,-\sigma_m t^{\beta-\beta_m},-\gamma_\zeta t^{\beta})||(w_1,e_\zeta)| \nonumber\\
&\leqslant \frac{C_{T,\vec{\beta},\vec{\sigma}}}{1+\gamma_\zeta t^{\beta}}\left(|(w_0,e_\zeta)|\sum_{k=0}^{m}\sigma_k t^{\beta-\beta_k}+|(w_1,e_\zeta)|\sum_{k=0}^{m}\sigma_k t^{\beta-\beta_k+1}\right),\quad 0<t\leqslant T.\label{multi-estimate-wave}
\end{align}
Thus
\begin{align*}
\|w(t)\|_{\mathcal{H}_\mfL^2}^2&=\sum_{\zeta\in I}(1+\gamma_\zeta)^2|w_\zeta(t)|^2 \\
&\leqslant C_{T,\vec{\beta},\vec{\sigma}}^2 \left(\sum_{k=0}^{m}\sigma_k t^{\beta-\beta_k}\right)^2\sum_{\zeta\in I}\frac{(1+\gamma_\zeta)^2}{(1+\gamma_\zeta t^{\beta})^2}\big(|(w_0,e_\zeta)|^2+t^2|(w_1,e_\zeta)|\big) \\
&\leqslant 
\left\{
\begin{array}{rccl}
&\displaystyle C_{T,\vec{\beta},\vec{\sigma}}^2 \left(\sum_{k=0}^{m}\sigma_k t^{\beta-\beta_k}\right)^2 \big(\|w_0\|_{\mathcal{H}_\mfL^2}^2+t^2 \|w_1\|_{\mathcal{H}_\mfL^2}^2 \big),\,\, 0\in\sigma(\mfL), \\
&\displaystyle C_{T,\vec{\beta},\vec{\sigma}}^2 \left(\sum_{k=0}^{m}\sigma_k t^{\beta-\beta_k}\right)^2 (1+t^{\beta})^{-2}\big(\|w_0\|_{\mathcal{H}_\mfL^2}^2+t^2 \|w_1\|_{\mathcal{H}_\mfL^2}^2\big),\,\,0\notin\sigma(\mfL), \\
&\displaystyle C_{T,\vec{\beta},\vec{\sigma}}^2\left(\sum_{k=0}^{m}\sigma_k t^{\beta-\beta_k}\right)^2\bigg(\|w_0\|_{\mathcal{H}_\mfL^2}^2+(1+t)^2\|w_1\|_{\mathcal{H}_{\mfL}^{\frac{2(\beta-1)}{\beta}}}^2\bigg), \\
&\displaystyle C_{T,\vec{\beta},\vec{\sigma}}^2 (1+t^{-\beta})^2\left(\sum_{k=0}^{m}\sigma_k t^{\beta-\beta_k}\right)^2\big(\|w_0\|_{\mathcal{H}}^2+t^2 \|w_1\|_{\mathcal{H}}^2\big).
\end{array}
\right.
\end{align*}
The proof is complete.
\end{proof}

\begin{rem}
The analysis made in this section just considered homogeneous equations. Nevertheless, the same analysis will work for nonhomogeneous type under some suitable conditions on the source function.    
\end{rem}

\begin{rem}
Everywhere in this section we considered $\mfL:\text{Dom}(\mfL)\subset\mathcal{H}\to\mathcal{H}$ to be a positive linear operator densely defined in a separable Hilbert space $\mathcal{H}$. Notice that the assumption on densely defined operator allows us to think about operators of unbounded type. In fact, if we consider an operator $L:\mathcal{H}\to\mathcal{H}$ defined in the whole space (assume complex Hilbert space) and positive, then $L$ is bounded (Hellinger-Toeplitz Theorem).         
\end{rem}

\section{Heat and wave type equations on a graded Lie group}\label{ss-graded}

In this section we analyse the heat and wave type equations on a graded Lie group by using the tools of the Fourier analysis of the group. This work will complement the previous section as well as the works \cite{WRR,palmieri,RNT}.   

In the first part we prove the well-posedeness of $\mathcal{R}$-wave type equations and give Sobolev-norm estimates of the solution. In the second part, we focus on $F(\mathcal{R})$-heat type equations, where $F$ is some suitable function. In this case we also prove $L^p-L^q$ $(1\leqslant p\leqslant 2\leqslant q<+\infty)$ estimates for the solutions. We finish this section by studying multi-term heat and wave type equations in our setting.   

Let us point out the since the spectrum of the Rockland operator $\mathcal{R}$ is continuous, the setting here does not fall into the one in Section \ref{explicit}. However, we can make such a reduction by using the group Fourier transform, and in particular by considering the operator $\pi(\mathcal{R})$ that has discrete spectrum.

Let us fist recall the notion of the  \textit{Sobolev spaces} in this setting which will be used frequently in this section. For a detailed discussion on the Sobolev spaces in the setting of a graded Lie group, see \cite[Subsection 4.4.1]{FR16} and \cite{FR17}. 

Let $\mathcal{R}$ be a positive Rockland operator of homogeneous degree $\nu$, and let $s\in\mathbb{R}$. The \textit{homogeneous Sobolev space $\dot{L}^{2}_{s}(G)$} is the subspace of tempered distributions $\mathcal{S}'(G)$ obtained by the completion of $ \mathcal{S}(G)$ with respect to the (homogeneous) \textit{Sobolev norm}
\[
\|u\|_{\dot{L}^{2}_{s}(G)}:=\|\mathcal{R}_{2}^{\frac{s}{\nu}}u\|_{L^2(G)}\,,\quad u \in \mathcal{S}(G),
\]
where the operator $\mathcal{R}_2$ stands for the self-adjoint extension of $\mathcal{R}$ on $L^2(G)$. Under the same considerations, the \textit{(nonhomogeneous) Sobolev space $L^{2}_{s}(G)$} is the subspace of tempered distributions $\mathcal{S}'(G)$ obtained by the completion of $ \mathcal{S}(G)$ with respect to the (nonhomogeneous) \textit{Sobolev norm}
\[
\|u\|_{L^{2}_{s}(G)}:=\|(I+\mathcal{R}_{2})^{\frac{s}{\nu}}u\|_{L^2(G)}\,,\quad u \in \mathcal{S}(G)\,.
\]In the sequel, we will use the notation $\mathcal{R}$ for its self-adjoint extension on $L^2(G)$ as well. 

It is worth mentioning that the Sobolev spaces do not depend on the specific choice of $\mathcal{R}$, in the sense that, different choices of the latter produce equivalent norms, see \cite[Proposition 4.4.20]{FR16}.

\subsection{$\mathcal{R}$-wave type equations}

In this subsection we analyse the solution of the following equation:
\begin{equation}\label{wave}
\left\{ \begin{split}
\prescript{C}{}\partial_{t}^{\beta}w(t,x)+\mathcal{R}w(t,x)&=0,\quad t>0,\quad x \in G,  \\
w(t,x)|_{_{_{t=0}}}&=w_0(x)\,, \\
\partial_t w(t,x)|_{_{_{t=0}}}&=w_1(x)\,.
\end{split}
\right.
\end{equation}

In \eqref{wave} the Rockland operator $\mathcal{R}$ is regarded to be positive, $^{C}\partial_{t}^{\alpha}$ is the Dzhrbashyan-Caputo fractional derivative from \eqref{djerbashian} and $1< \beta <2$. The next theorem shows that the solution to the above equation is related with the Mittag-Leffler propagator defined in \eqref{Mittag-propagator}.  
\begin{thm}
    \label{THM:graded,wtyp} Let $G$ be a graded Lie group, and let $\mathcal{R}$ be a positive Rockland operator on $G$ of homogeneous degree $\nu$. On $G$ we consider the Cauchy problem \eqref{wave}. We have:
   
    \begin{enumerate}[label=(\alph*)]
        \item \label{itm:1est} for any $t>0$ the continuous in $t$ solution to the problem \eqref{wave} is explicitly given by 
        \[
w(t,x)=E_\beta(-t^{\beta}\mathcal{R})w_0(x)+tE_{\beta,2}(-t^{\beta}\mathcal{R})w_1(x)\,,\quad x \in G\,;
        \]
\item \label{itm:2est} if $(w_0,w_1) \in L^2(G)\times L^2(G)$, then the continuous in $t$ solution $w$ is unique and satisfies the estimate
\[
    \|w(t,\cdot)\|_{L^2(G)} \lesssim \|w_0\|_{L^2(G)}+t\|w_1\|_{L^2(G)}\,,\quad \text{for all} \quad t>0;
\]
\item \label{itm:3est}if $(w_0,w_1) \in L^{2}_{s}(G)\times L^{2}_{s}(G)$, then the continuous in $t$ solution $w$ is unique and satisfies the Sobolev-norm estimate
\[
\|w(t,\cdot)\|_{L^{2}_{s}(G)}\lesssim \|w_0\|_{L^{2}_{s}(G)}+t\|w_1\|_{L^{2}_{s}(G)}\,,\quad{\text{for any}}\quad s\in\mathbb{R},
\]
which, in particular, implies estimate in \ref{itm:2est} for $s=0$;
\item \label{itm:4est} if $(w_0,w_1) \in L^{2}_{s}(G)\times L^{2}_{s-\nu/\beta}(G)$, then the continuous in $t$ solution $w$ is unique and satisfies the Sobolev-norm estimate
\[
\|w(t,\cdot)\|_{L^{2}_{s}(G)}\lesssim \|w_0\|_{L^{2}_{s}(G)}+(1+t)\|w_1\|_{L^{2}_{s-\nu/\beta}(G)},\quad s\geqslant \frac{\nu}{\beta}.
\]
\item \label{itm:6est} Finally, we also obtain that 
\begin{align*}
     \|\partial_t w(t,\cdot)\|_{L^2(G)}&\lesssim 
\left\{
\begin{array}{rccl}
& t^{\beta-1}\|w_0\|_{\dot{L}_{\nu}^2(G)}+\|w_1\|_{L^2(G)},\,\quad (w_0,w_1)\in \dot{L}_{\nu}^2(G)\times L^2(G), \\
& \|w_0\|_{\dot{L}_{\nu/\beta}^2(G)}+\|w_1\|_{L^2(G)},\,\quad (w_0,w_1)\in \dot{L}_{\nu/\beta}^2(G)\times L^2(G), \\
& t^{-1}\|w_0\|_{L^2(G)}+\|w_1\|_{L^2(G)},\,\quad (w_0,w_1)\in L^2(G)\times L^2(G).
\end{array}
\right.     
 \end{align*}

    \end{enumerate}
\end{thm}
\begin{proof} We apply the group Fourier transform with respect to $x$ for the equation \eqref{wave} and get
    \begin{equation}\label{wave-f}
\left\{ \begin{split}
\prescript{C}{}\partial_{t}^{\beta}\widehat{w}(t,\pi)+\pi(\mathcal{R})\widehat{w}(t,\pi)&=0,\quad t>0\,\quad \pi \in \widehat{G},   \\
\widehat{w}(t,\pi)|_{_{_{t=0}}}&=\widehat{w}_0(\pi)\,, \\
\partial_t \widehat{w}(t,\pi)|_{_{_{t=0}}}&=\widehat{w}_1(\pi)\,.
\end{split}
\right.
\end{equation}
The latter in view of \eqref{repr.pr} yields the infinite dimensional system
 \begin{equation}\label{wave-f-m}
\left\{ \begin{split}
\prescript{C}{}\partial_{t}^{\beta}\widehat{w}(t,\pi)_{i,j}+\pi_{i}^{2}\widehat{w}(t,\pi)_{i,j}&=0,\quad t>0\,,  \\
\widehat{w}(t,\pi)_{i,j}|_{_{_{t=0}}}&=\widehat{w}_0(\pi)_{i,j}\,, \\
\partial_t \widehat{w}(t,\pi)_{i,j}|_{_{_{t=0}}}&=\widehat{w}_1(\pi)_{i,j}\,.
\end{split}
\right.
\end{equation}
Let us now fix $(i,j) \in \mathbb{N} \times \mathbb{N}$ and $\pi \in \widehat{G}$. The  system \eqref{wave-f-m} is now decoupled and under this consideration each equation described by \eqref{wave-f-m} contains only time dependent functions. We can then apply the Laplace transform in $t$ and get
\begin{equation}\label{wave-f-m-ltr}
\left\{ \begin{split}
s^{\beta}\widetilde{\widehat{w}}(s,\pi)_{i,j}-s^{\beta-1}\widehat{w}_0(\pi)_{i,j}-s^{\beta-2}\widehat{w}_1(\pi)_{i,j}+\pi_{i}^{2}\widetilde{\widehat{w}}(s,\pi)_{i,j}&=0,\quad s>0\,,  \\
\widehat{w}(t,\pi)_{i,j}|_{_{_{t=0}}}&=\widehat{w}_0(\pi)_{i,j}\,, \\
\partial_t \widehat{w}(t,\pi)_{i,j}|_{_{_{t=0}}}&=\widehat{w}_1(\pi)_{i,j}\,,
\end{split}
\right.
\end{equation}
which in turn implies
\[
\widetilde{\widehat{w}}(s,\pi)_{i,j}=\frac{s^{\beta-1}}{s^{\beta}+\pi_{i}^{2}}\widehat{w}_{0}(\pi)_{i,j}+\frac{s^{\beta-2}}{s^{\beta}+\pi_{i}^{2}}\widehat{w}_{1}(\pi)_{i,j}\,,\quad s>0\,.
\]
An application of the inverse Laplace transform, see e.g. \cite[Theorem 2.1]{new-mittag-add} yields
\begin{equation}\label{wide.w}
\widehat{w}(t,\pi)_{i,j}=E_\beta(-\pi_{i}^{2}t^{\beta})\widehat{w}_0(\pi)_{i,j}+t E_{\beta,2}(-\pi_{i}^{2}t^{\beta})\widehat{w}_{1}(\pi)_{i,j}\,,
\end{equation}
and after summation over $i,j$, we get the following representation of the Fourier transform of the solution
\begin{align*}
\widehat{w}(t,\pi)&=\sum_{k=0}^{+\infty}\frac{(-t^{\beta})^k}{\Gamma(\beta k+1)}\big(\pi(\mathcal{R})\big)^k \widehat{w_0}(\pi)+t\sum_{k=0}^{+\infty}\frac{(-t^{\beta})^k}{\Gamma(\beta k+2)}\big(\pi(\mathcal{R})\big)^k \widehat{w_1}(\pi) \\
&=\sum_{k=0}^{+\infty}\frac{(-t^{\beta})^k}{\Gamma(\beta k+1)}\widehat{\mathcal{R}^k w_0}(\pi)+t\sum_{k=0}^{+\infty}\frac{(-t^{\beta})^k}{\Gamma(\beta k+2)}\widehat{\mathcal{R}^k w_1}(\pi).
\end{align*}

On the other hand by the Fourier inversion formula \eqref{Four.inv.for} the solution $w$ can be represented by 
\[
w(t,x)=\int_{\widehat{G}}
\textnormal{Tr}[\pi(x)\widehat{w}(t,\pi)]\,d\mu(\pi)\,,
\]
or, by the above
\begin{align*}
w(t,x)&=\sum_{k=0}^{+\infty}\frac{(-t^{\beta})^k}{\Gamma(\beta k+1)}\int_{\widehat{G}}
\textnormal{Tr}[\pi(x)\widehat{\mathcal{R}^k w_0}(\pi)]\,d\mu(\pi)\, \\
&\hspace{3cm}+t\sum_{k=0}^{+\infty}\frac{(-t^{\beta})^k}{\Gamma(\beta k+2)}\int_{\widehat{G}}
\textnormal{Tr}[\pi(x)\widehat{\mathcal{R}^k w_1}(\pi)]\,d\mu(\pi)\, \\
&=\sum_{k=0}^{+\infty}\frac{(-t^{\beta}\mathcal{R})^k}{\Gamma(\beta k+1)}w_0(x)+t\sum_{k=0}^{+\infty}\frac{(-t^{\beta}\mathcal{R})^k}{\Gamma(\beta k+2)}w_1(x).
\end{align*}
The last implies that 
\[
w(t,x)=E_\beta(-t^{\beta}\mathcal{R})w_0(x)+tE_{\beta,2}(-t^{\beta}\mathcal{R})w_1(x)\,,
\]
and we have proved \ref{itm:1est}.
Now, to achieve an estimate for the $L^2$-norm in $x$ of $w(\cdot,t)$ we first combine the expression \eqref{wide.w} with estimates for the Mittag--Leffler functions (see \cite[Theorem 1.6]{page 35}) and obtain 
\begin{eqnarray}\label{pg35}
|\widehat{w}(t,\pi)_{i,j}|\leqslant C \frac{1}{1+\pi_{i}^{2}t^{\beta}}|\widehat{w}_{0}(\pi)_{i,j}|+C \frac{t}{1+\pi_{i}^{2}t^{\beta}}|\widehat{w}_1(\pi)_{i,j}|\,.
\end{eqnarray}
The estimate \eqref{pg35} holds true uniformly in $\pi \in \widehat{G}$ and all $i,j$. Recall that for a Hilbert-Schmidt operator $A$ on $\mathcal{H}$ one has 
\[
\|A\|^{2}_{\textnormal{HS}(\mathcal{H})}=\sum_{k,l}|\langle A \varphi_{k},\varphi_{l}\rangle|^2\,,
\]
where $\{\varphi_1,\varphi_2,\cdots\}$ is some orthonormal basis of $\mathcal{H}$. Now, observe that by \eqref{pg35} we get
\[
|\widehat{w}(t,\pi)_{i,j}|\lesssim |\widehat{w}_{0}(\pi)_{i,j}|+ t|\widehat{w}_1(\pi)_{i,j}|\,,  \quad \text{for all} \quad t>0\,,
\]
where the latter yields, after summation over $i,j$, the following estimate for the Hilbert-Schmidt norm of the operator $\widehat{w}(t,\pi)$:
\[
\|\widehat{w}(t,\pi)\|_{\textnormal{HS}(\mathcal{H}_{\pi})}\lesssim \|\widehat{w}_0(\pi)\|_{\textnormal{HS}(\mathcal{H}_{\pi})}+t\|\widehat{w}_1(\pi)\|_{\textnormal{HS}(\mathcal{H}_{\pi})}\,.
\]

Therefore, integrating the above over $\widehat{G}$ against the Plancherel measure and using the Plancherel formula \eqref{plan.gr} we get 
\begin{equation}
    \label{after.pl}
    \|w(t,\cdot)\|_{L^2(G)}\lesssim \|w_0\|_{L^2(G)}+t\|w_1\|_{L^2(G)}\,, \quad \text{for all} \quad t>0\,,
\end{equation}
and we have proved \ref{itm:2est}.

Now, notice that for any $i,j \in \mathbb{N}$ using \eqref{pg35} we have 
\begin{align}\label{withpi}
(1+\pi_{i}^2)^{s/\nu}&|\widehat{w}(t,\pi)_{i,j}|\lesssim  \frac{(1+\pi_{i}^2)^{s/\nu}}{1+\pi_{i}^{2}t^{\beta}}\bigg(|\widehat{w}_{0}(\pi)_{i,j}|+ t|\widehat{w}_1(\pi)_{i,j}|\bigg)\nonumber\\
&\lesssim  
\left\{
\begin{array}{rccl}
&(1+\pi_{i}^2)^{s/\nu}\bigg(|\widehat{w}_{0}(\pi)_{i,j}|+ t|\widehat{w}_1(\pi)_{i,j}|\bigg),\quad s\in\mathbb{R},\nonumber \\
&(1+\pi_{i}^2)^{s/\nu}|\widehat{w}_{0}(\pi)_{i,j}|+\bigg(t+\pi_i^{2\left(\frac{s}{\nu}-\frac{1}{\beta}\right)}\bigg)|\widehat{w}_1(\pi)_{i,j}|,\quad s\geqslant \frac{\nu}{\beta},
\end{array}
\right.
\end{align}
where for the second term of the second estimate we have considered the supremum of the function $h(t)=\frac{t}{1+\pi_{i}^{2}t^{\beta}}$.
Therefore, another application of the Plancherel formula yields
\begin{align*}
\|w(t,\cdot)\|_{L^{2}_{s}(G)}&\lesssim  
\left\{
\begin{array}{rccl}
&\|w_0\|_{L^{2}_{s}(G)}+t\|w_1\|_{L^{2}_{s}(G)},\quad s\in\mathbb{R}, \\
&\|w_0\|_{L^{2}_{s}(G)}+t\|w_1\|_{L^{2}(G)}+\|w_1\|_{\dot{L}^{2}_{s-\nu/\beta}(G)},\quad s\geqslant \frac{\nu}{\beta},
\end{array}
\right. \\
&\lesssim  
\left\{
\begin{array}{rccl}
&\|w_0\|_{L^{2}_{s}(G)}+t\|w_1\|_{L^{2}_{s}(G)},\quad s\in\mathbb{R}, \\
&\|w_0\|_{L^{2}_{s}(G)}+(1+t)\|w_1\|_{L^{2}_{s-\nu/\beta}(G)},\quad s\geqslant \frac{\nu}{\beta},
\end{array}
\right.
\end{align*}
which gives \ref{itm:3est} and \ref{itm:4est}. 

It remains to prove the estimate \ref{itm:6est}. To this end we differentiate with respect to $t$ the formula \eqref{wide.w} and get
\begin{eqnarray*}
    \partial_t \widehat{w}(t,\pi)_{i,j} & = & \sum_{l=1}^{+\infty}\frac{(-\pi_{i}^{2})^{l}\beta l t^{\beta l-1}}{\Gamma(\beta l+1)}\widehat{w}_0(\pi)_{i,j}+\sum_{l=1}^{+\infty}\frac{(-\pi_{i}^{2})^{l}(\beta l+1) t^{\beta l}}{\Gamma(\beta l+2)}\widehat{w}_1(\pi)_{i,j}\\
    &= & -\pi_{i}^{2}t^{-1+\beta}E_{\beta,\beta}(-\pi_{i}^{2}t^{\beta})\widehat{w}_0(\pi)_{i,j}+E_{\beta}(-\pi_{i}^{2}t^{\beta})\widehat{w}_1(\pi)_{i,j}\,,
\end{eqnarray*}
where for the second equality we have used formula \eqref{multivariateML}. Arguing as we did earlier we can estimate the elements of the infinite matrix representation $|\partial_t \widehat{w}(t,\pi)_{i,j}|$ as follows:
 \begin{align*}
     |\partial_t \widehat{w}(t,\pi)_{i,j}| & \lesssim t^{\beta-1}\pi_{i}^{2}|E_{\beta,\beta}(-\pi_{i}^{2}t^{\beta})| |\widehat{w}_0(\pi)_{i,j}|+|E_{\beta}(-\pi_{i}^{2}t^{\beta})| |\widehat{w}_{1}(\pi)_{i,j}|\nonumber\\
     & \lesssim  \frac{t^{\beta-1}\pi_{i}^{2}}{1+\pi_{i}^{2}t^{\beta}}|\widehat{w}_0(\pi)_{i,j}|+\frac{1}{1+\pi_{i}^{2}t^{\beta}}|\widehat{w}_{1}(\pi)_{i,j}|\nonumber\\
     & \lesssim  \frac{t^{\beta-1}\pi_{i}^{2}}{1+\pi_{i}^{2}t^{\beta}}|\widehat{w}_0(\pi)_{i,j}|+|\widehat{w}_{1}(\pi)_{i,j}|\, \\
&\lesssim 
\left\{
\begin{array}{rccl}
& t^{\beta-1}\pi_i^{2}|\widehat{w}_0(\pi)_{i,j}|+|\widehat{w}_{1}(\pi)_{i,j}|\, \\
& \pi_i^{2/\beta}|\widehat{w}_0(\pi)_{i,j}|+|\widehat{w}_{1}(\pi)_{i,j}|\, \\
& t^{-1}|\widehat{w}_0(\pi)_{i,j}|+|\widehat{w}_{1}(\pi)_{i,j}|,
\end{array}
\right.     
 \end{align*}
where for the second estimate we have calculated the supermum of the function $g(t)= \frac{t^{\beta-1}\pi_{i}^{2}}{1+\pi_{i}^{2}t^{\beta}}$ and we arrive easily at the estimate \ref{itm:6est}. 
\end{proof}

\subsection{$L^p$-$L^q$ estimates for $F(\mathcal{R})$-heat type equations}\label{heat-section}

We investigate the $L^p$-$L^q$ estimates and asymptotic time-behaviour of the following $F(\mathcal{R})$-heat type equation on a graded Lie group $G$:  
\begin{equation}\label{HeatTypeEquationG}
\begin{split}
^{C}\partial_{t}^{\alpha}w(t,x)+F(\mathcal{R})w(t,x)&=0, \quad t>0,\,\, x\in G, \\
w(t,x)|_{_{_{t=0}}}&=w_0(x),
\end{split}
\end{equation}
where $^{C}\partial_{t}^{\alpha}$ is the Dzhrbashyan-Caputo fractional derivative from \eqref{djerbashian}, $\mathcal{R}$ is a positive Rockland operator, $F:[0,\infty) \rightarrow [0,\infty)$ is an increasing function such that $\displaystyle\lim_{s\to+\infty}F(s)=+\infty$ and $0<\alpha<1$. Here we will not treat the case $\alpha=1$ since it is already known, see \cite[Section 6]{RR2020}. In particular, in the next theorem we show  the existence of a unique continuous solution of equation \eqref{HeatTypeEquationG} by using the group Fourier transform. For the solution, we also provide the $L^p-L^q$ estimates for $1\leqslant p\leqslant 2\leqslant q<+\infty$. Moreover, we establish time-decay estimates for the general case of a graded Lie group, while in the examples on particular cases of such groups that follow, the decay is sharp.    

To understand the expressions \eqref{need} and \eqref{asymtotic-trace}, we refer to the discussion before Theorem \ref{thm1.5}, or to the work \cite{RR2020}.

\begin{thm}\label{Main-heat}
Let $F:[0,\infty) \rightarrow [0,\infty)$ be an increasing function such that $\displaystyle\lim_{s\to+\infty}F(s)=+\infty$. Let $G$ be a graded Lie group, $\mathcal{R}$ be a positive Rockland operator on $G$ of homogeneous degree $\nu$, $0<\alpha\leqslant 1$ and $1\leqslant p\leqslant 2\leqslant q<+\infty$. Then there exists a unique continuous solution to the $F(\mathcal{R})$-heat type equation \eqref{HeatTypeEquationG} that is given by
\[
w(t,x)=E_\alpha(-t^{\alpha}F(\mathcal{R}))w_0(x),\quad t>0,\,\,x\in G\,,
\]
where the propagator can be precisely expressed as
\[
E_\alpha(-t^{\alpha}F(\mathcal{R})) = \sum_{k=0}^{+\infty} \frac{(-t^{\alpha}F(\mathcal{R}))^k}{\Gamma(\alpha k+1)}.
\]
Additionally, if $w_0 \in L^p(G)$, and  
\begin{equation}\label{need}
\sup_{t>0}\sup_{s>0}[\tau\big(E_{(0,s)}(\mathcal{R})\big)]^{\frac{1}{p}-\frac{1}{q}}E_\alpha(-t^{\alpha}F(s))<+\infty,   
\end{equation}
then there exits a unique solution $w\in\mathcal{C}\big([0,+\infty);L^q(G)\big).$ 

In particular, if for some $\gamma>0$ we have 
\begin{equation}\label{asymtotic-trace}
\tau\big(E_{(0,s)}(\mathcal{R})\big)\lesssim s^{\gamma},\quad s\to+\infty,\quad \text{and}\quad F(s)=s,
\end{equation}
then \eqref{need} is satisfied for any $1<p\leqslant 2\leqslant q<+\infty$ such that $\frac{1}{\gamma}>\frac{1}{p}-\frac{1}{q}$, and we get the following time decay rate for the solution of equation \eqref{HeatTypeEquationG}: 
\[
\|w(t,\cdot)\|_{L^q(G)}\leqslant C_{\alpha,\gamma,p,q}t^{-\alpha\gamma\left(\frac{1}{p}-\frac{1}{q}\right)}\|w_0\|_{L^p(G)},    
\]
where $C_{\alpha,\gamma,p,q}$ does not depend on $w_0$ and $t>0.$

\end{thm}
\begin{proof}
We first take the group Fourier transform in equation \eqref{HeatTypeEquationG} with respect to the variable $x$ for all $\pi\in\widehat{G}$ and get
\begin{align}\label{fouriere}
\begin{split}
^{C}\partial_{t}^{\alpha}\widehat{w}(t,\pi)+\pi(F(\mathcal{R}))\widehat{w}(t,\pi)&=0, \quad t>0, \\
\widehat{w}(t,\pi)|_{_{_{t=0}}}&=\widehat{w_0}(\pi).
\end{split}
\end{align}
Taking into account the infinitesimal representation \eqref{repr.pr} of $\pi(\mathcal{R})$, the functional calculus allows the latter equation to be seen  componentwise as an infinite system of equations of the form

\begin{equation}\label{diagon}
\begin{split}
^{C}\partial_{t}^{\alpha}\widehat{w}(t,\pi)_{ij}+F(\pi_{i}^{2})\widehat{w}(t,\pi)_{ij}&=0,\quad t>0, \\
\widehat{w}(t,\pi)_{ij}|_{_{_{t=0}}}&=\widehat{w_0}(\pi)_{ij},
\end{split}
\end{equation}
where we are considering any $i,j\in\mathbb{N}$ and any $\pi\in \widehat{G}.$ To solve the system  given by the infinite matrix equation \eqref{diagon}, we decouple the system by fixing an index $(i,j)$. Then each equation given by \eqref{diagon} consists of functions only in the time-variable $t$ and will be treated independently. Thus, for fixed $(i,j)$ we apply the Laplace transform in $t$ and obtain
\[
\begin{split}
u^{\alpha}\widetilde{\widehat{w}}(u,\pi)_{ij}-u^{\alpha-1}\widehat{w_0}(\pi)_{ij}+F(\pi_{i}^{2}) \widetilde{\widehat{w}}(u,\pi)_{ij}&=0, \quad u>0.
\end{split}
\]
Hence
\[
\widetilde{\widehat{w}}(u,\pi)_{ij}=\frac{u^{\alpha-1}}{u^{\alpha}+F(\pi_{i}^{2})}\widetilde{\widehat{w_0}}(\pi)_{ij}, \quad u>0.
\]
An application of the inverse Laplace transform (see e.g.  \cite[Theorem 2.1]{new-mittag-add}) yields
\[
\widehat{w}(t,\pi)_{ij}=E_{\alpha}(-F(\pi_{i}^{2}) t^{\alpha})\widehat{w_0}(\pi)_{ij}\,,
\]
or after a summation over $i,j$,
\[
\widehat{w}(t,\pi)=E_{\alpha}(-\pi(F(\mathcal{R})) t^{\alpha})\widehat{w_0}(\pi)\,,
\]
which in turn implies that
\begin{equation}
    \label{Tr}
    \int_{\widehat{G}}\textnormal{Tr}[\pi(x)\widehat{w}(t,\pi)]\,{\rm d}\mu(\pi)=  \int_{\widehat{G}}\textnormal{Tr}[\pi(x)E_{\alpha}(-\pi(F(\mathcal{R})) t^{\alpha})\widehat{w_0}(\pi)]\,{\rm d}\mu(\pi)\,,
\end{equation}
where the left hand side of \eqref{Tr} equals by \eqref{Four.inv.for} to $w(t,x)$. Hence 
\[
w(t,x)=E_{\alpha}(-(F(\mathcal{R}) t^{\alpha})w_0(x)\,.
\]
Now, by Theorem 5.1 in \cite{RR2020} we have
\begin{align}\label{RR1}
    \|w(t,\cdot)\|_{L^q(G)}&=\|E_{\alpha}(-t^{\alpha}F(\mathcal{R})w_0(\cdot)\|_{L^q(G)} \nonumber \\
    &\lesssim \|E_{\alpha}(-t^{\alpha}F(\mathcal{R}))\|_{L^{r,\infty}(VN_{R}(G))}\|w_0\|_{L^p(G)}\,,
\end{align}
where $r,p,q$ satisfy the relation $\frac{1}{r}=\frac{1}{p}+\frac{1}{q}$, and the Lorentzian norm is given by \cite[Theorem 6.1]{RR2020}
\begin{equation}
    \label{Lor.norm}
  \|E_{\alpha}(-t^{\alpha}F(\mathcal{R}))\|_{L^{r,\infty}(VN_{R}(G))}=\sup_{s>0}  \left[\tau\left( E_{(0,s)}(\mathcal{R}\right)\right]^{\frac{1}{r}}E_{\alpha}(-t^{\alpha}F(s))<+\infty\,,
\end{equation}
and is finite by the assumption \eqref{need} in the hypothesis. 
Here we are using that $E_\alpha(-s)$, $s\geqslant0$, is completely monotonic function \cite{Pollard} and is such that $E_\alpha(0)=1$ and $\displaystyle\lim_{s\to+\infty}E_\alpha(-s)=0$ by the uniform estimate given in \cite[Theorem 4]{Mittag-bounded}. Now, by the assumptions on $F$ , the latter implies that $E_\alpha(-t^{\alpha}F(s))$ is monotonically decreasing for any $s\geqslant0$ , and $\displaystyle\lim_{s\to+\infty}E_\alpha(-t^{\alpha}F(s))=0$. 

Suppose now that the conditions in \eqref{asymtotic-trace} are satisfied. Then, for $F(s)=s$  using the uniform estimate for the Mittag-Leffler function as in \eqref{MLest2}, and the expression \eqref{Lor.norm} we get
\begin{equation}\label{Lor.norm.1}
     \|E_{\alpha}(-t^{\alpha}\mathcal{R})\|_{L^{r,\infty}(VN_{R}(G))}=\sup_{s>0} s^{\frac{\gamma}{r}}\frac{1}{1+\frac{t^{\alpha}s}{\Gamma(1+\alpha)}}\,.
\end{equation}
For $g(s)=s^{\frac{\gamma}{r}}\frac{1}{1+\frac{t^{\alpha}s}{\Gamma(1+\alpha)}}$, using standard analytical arguments, we see that the supremum of $g$ is attained at $s^*=\frac{\gamma \Gamma(1+\alpha)}{r\left(1-\frac{\gamma}{r} \right)}t^{-\alpha}$, and we have $g(s^*)=C_{a,\gamma,p,q}t^{-\frac{\alpha \gamma}{r}}$. Hence we have
$$\|w(t,\cdot)\|_{L^q(G)}\leqslant C_{\alpha,\gamma,p,q}t^{-\alpha\gamma\left(\frac{1}{p}-\frac{1}{q}\right)}\|w_0\|_{L^p(G)}.$$ 
The proof of Theorem \ref{Main-heat} is then complete. 
\end{proof}

\subsubsection{Examples}\label{examples} Below we present  several examples where the trace of the spectral projections is already known. In these cases Theorem \ref{Main-heat} can be applied, and one can check that time decay obtained is optimal, if comparing the latter with existing results.  

\
\begin{ex}{The Euclidean space $\mathbb{R}^n$.}
We consider the heat type equation 
\begin{equation}\label{rn-e}
\begin{split}
^{C}\partial_{t}^{\alpha}w(t,x)-\Delta_{\mathbb{R}^n} w(t,x)&=0, \quad t>0,\,\, x\in \mathbb{R}^n, \,\, 0<\alpha<1, \\
w(t,x)|_{_{_{t=0}}}&=w_0(x),\quad w_0\in L^p(\mathbb{R}^n),\quad 1<p\leqslant 2,
\end{split}
\end{equation}
where $\Delta_{\mathbb{R}^n}=\sum_{i=1}^n \partial_{x_i}^2$ is the classical Laplacian operator on $\mathbb{R}^n$. It is already known that the trace of the spectral projections $E_{(0,s)}(\Delta_{\mathbb{R}^n})$ has the following asymptotic behavior \cite{RR18}:
\[
\tau\big(E_{(0,s)}(\Delta_{\mathbb{R}^n})\big)\lesssim s^{n/2},\quad s\to+\infty.
\]
Hence,  Theorem \ref{Main-heat} is applicable, and we get the following {\it sharp time decay rate} for the solution of equation \eqref{rn-e}:\[
\|w(\cdot,t)\|_{L^q(\mathbb{R}^n)}\leqslant C_{\alpha,n,p,q}t^{-\frac{\alpha n}{2}\left(\frac{1}{p}-\frac{1}{q}\right)} \|w_0\|_{L^p(\mathbb{R}^n)},  
\]
for all $1<p\leqslant 2\leqslant q<+\infty$ with $\frac{2}{n}>\frac{1}{p}-\frac{1}{q}$. Indeed, this decay is sharp since it coincides with the result given in \cite[Theorem 3.3, item (i)]{uno2}.
\end{ex}
Below we would like to mention a particular analogous result which can be obtained in a compact Lie group for heat type equations since in this section we are extended the results of \cite[Section 3]{WRR} to a graded Lie group.  
\begin{ex} Compact Lie groups. We now study the following heat type equation by using the sub-Laplacian $\Delta_{sub}$ on a compact Lie group $G$: 
\begin{align*}
\begin{split}
^{C}\partial_{t}^{\alpha}w(t,x)-\Delta_{sub}w(t,x)&=0, \quad t>0,\,\, x\in G, \,\, 0<\alpha\leqslant 1, \\
w(t,x)|_{_{_{t=0}}}&=w_0(x),\quad w_0\in L^p(G),\quad 1<p\leqslant 2.
\end{split}
\end{align*}
By \cite{[35]}, we have that the trace of the spectral projections $E_{(0,s)}(-\Delta_{sub})$ has the following asymptotic behavior:
\[
\tau\big(E_{(0,s)}(-\Delta_{sub})\big)\lesssim s^{Q/2},\quad s\to+\infty,
\]
where $Q$ is the Hausdorff dimension of $G$ with respect to the control distance generated by the sub-Laplacian. So, by \cite[Theorem 4]{WRR} (see also \cite[Example 1]{WRR}), we have the existence, uniqueness, the form, and the asymptotic behavior for the solution $w(t,x)$ as follows:
\[
\|w(t,\cdot)\|_{L^q(G)}\leqslant C_{\alpha,Q,p,q}t^{-\alpha Q/2\left(\frac{1}{p}-\frac{1}{q}\right)}\|w_0\|_{L^p(G)},\quad 2\leqslant q<+\infty,\quad \frac{2}{Q}>\frac{1}{p}-\frac{1}{q}. 
\]
 \end{ex}
 
More generally, we have the following result: 
\begin{ex}{Graded Lie groups.}\label{ex.graded}
Let us solve the following heat type equation by using any positive Rockland operator $\mathcal{R}$ of homogeneous degree $\nu$ on a graded Lie group $G$ of homogeneous dimension $Q$: 
\begin{align*}
\begin{split}
^{C}\partial_{t}^{\alpha}w(t,x)-\mathcal{R}w(t,x)&=0, \quad t>0,\,\, x\in G, \,\, 0<\alpha\leqslant 1, \\
w(t,x)|_{_{_{t=0}}}&=w_0(x),\quad w_0\in L^p(G),\quad 1<p\leqslant 2.
\end{split}
\end{align*}
By \cite{RR18}, we have that the trace of the spectral projections $E_{(0,s)}(\mathcal{R})$ has the following asymptotic behavior:
\[
\tau\big(E_{(0,s)}(\mathcal{R})\big)\lesssim s^{Q/\nu},\quad s\to+\infty.
\]
 So, by Theorem \ref{Main-heat} we have the existence, uniqueness, the form, and the asymptotic behavior for the solution $w(t,x)$ as follows:
\[
\|w(t,\cdot)\|_{L^q(G)}\leqslant C_{\alpha,Q,p,q}t^{-\alpha Q/\nu\left(\frac{1}{p}-\frac{1}{q}\right)}\|w_0\|_{L^p(G)},\quad 2\leqslant q<+\infty,\quad \frac{2}{Q}>\frac{1}{p}-\frac{1}{q}. 
\]
 Note that the sharpness of this decay for $\alpha=1$ is discussed in \cite{RR18}.

\end{ex}
 

A similar result can be also established on the Heisenberg group, see  \cite{RR18} or \cite[Section 7.3]{RR2020}.  

\subsection{Multi-term $\mathcal{R}$-heat type equations}\label{multi-section}

In this subsection we treat the following equation:  
\begin{equation}\label{Multi-HeatTypeEquationG}
	\left\{ \begin{aligned}
		\prescript{C}{}\partial_{t}^{\alpha_0}w(t,x)+\gamma_1\prescript{C}{}\partial_{t}^{\alpha_1}w(t,x)+\cdots+\gamma_m\prescript{C}{}\partial_{t}^{\alpha_m}w(t,x)+\mathcal{R}w(t,x)&=0,\,\,  \\
		w(t,x)|_{_{_{t=0}}}&=w_0(x),
	\end{aligned}
	\right.
\end{equation}
for $0<t\leqslant T<+\infty$ and $x\in G$, where $\mathcal{R}$ is a positive Rockland operator of homogeneous degree $\nu$ on $G$, $\gamma_i>0$ $(i=1,\ldots,m)$ and $0<\alpha_m<\alpha_{m-1}<\cdots<\alpha_1<\alpha_0\leqslant1$. 

\begin{thm}\label{multi-heat-thm}
Let $G$ be a graded Lie group, $\mathcal{R}$ be a positive Rockland operator on $G$ of homogeneous degree $\nu$ and let $0<\alpha_m<\alpha_{m-1}<\cdots<\alpha_1<\alpha_0\leqslant1$. Then there exists a unique continuous solution to equation \eqref{Multi-HeatTypeEquationG} given by
\begin{equation*}
w(t,x)=\sum_{k=0}^{m}t^{\alpha_0-\alpha_k}E_{(\alpha_0-\alpha_1,\ldots,\alpha_0-\alpha_m,\alpha_0),\alpha_0-\alpha_k+1}(-\gamma_1 t^{\alpha_0-\alpha_1},\ldots,-\gamma_m t^{\alpha_0-\alpha_m},-t^{\alpha_0}\mathcal{R}) w_0(x),\label{multi-solution}
		\end{equation*}
for any $0<t\leqslant T$ and $x\in G.$  
 Moreover we have the following Sobolev norm estimates:
 \begin{enumerate}
     \item For any $s\in\mathbb{R}$ and $w_0\in L^{2}_{s}(G)$ we have 
     \[
     \|w(t,\cdot)\|_{L^{2}_{s}(G)}\leqslant \displaystyle C_{T,\vec{\alpha},\vec{\gamma}}\left(\sum_{k=0}^{m}\gamma_k t^{\alpha_0-\alpha_k}\right)\|w_0\|_{L^{2}_{s}(G)},\quad 0<t\leqslant T.
     \]
     \item For any $s\geqslant \nu,$ and $w_0\in L^{2}_{s-\nu}(G)$ we have 
     \[
\|w(t,\cdot)\|_{L^{2}_{s}(G)}\leqslant \displaystyle C_{T,\vec{\alpha},\vec{\gamma}}\left(\sum_{k=0}^{m}\gamma_k t^{\alpha_0-\alpha_k}\right)(1+t^{-\alpha_0})\|w_0\|_{L^{2}_{s-\nu}(G)},\quad 0<t\leqslant T.
\]
 \end{enumerate}
\end{thm}
\begin{proof}
By applying the group Fourier transform to the equation \eqref{Multi-HeatTypeEquationG} we obtain
	 \begin{align*}
		\prescript{C}{}\partial_{t}^{\alpha_0}\widehat{w}(t,\pi)+\gamma_1\prescript{C}{}\partial_{t}^{\alpha_1}\widehat{w}(t,\pi)+\cdots+\gamma_m\prescript{C}{}\partial_{t}^{\alpha_m}\widehat{w}(t,\pi)+\pi(\mathcal{R})\widehat{w}(t,\pi)&=0,\quad \pi\in\widehat{G}, \\
		\widehat{w}(t,\pi)|_{_{_{t=0}}}&=\widehat{w_0}(\pi).
	\end{align*}
Hence
\begin{align*}
\prescript{C}{}\partial_{t}^{\alpha_0}\widehat{w}(t,\pi)_{i,j}+\gamma_1\prescript{C}{}\partial_{t}^{\alpha_1}\widehat{w}(t,\pi)_{i,j}+\cdots+\gamma_m\prescript{C}{}\partial_{t}^{\alpha_m}\widehat{w}(t,\pi)_{i,j}+\pi^2_i\widehat{w}(t,\pi)_{i,j}&=0, \\
		\widehat{w}(t,\pi)_{i,j}|_{_{_{t=0}}}&=\widehat{w_0}(\pi)_{i,j},
	\end{align*}
	for any $(i,j) \in \mathbb{N} \times \mathbb{N}$ and $\pi \in \widehat{G}$. Thus by the Laplace transform we have 
\begin{align*}	s^{\alpha_0}\widetilde{\widehat{w}}(s,\pi)_{i,j}-s^{\alpha_0-1}\widehat{w_0}(\pi)_{i,j}+\gamma_1 s^{\alpha_1}\widetilde{\widehat{w}}(s,\pi)_{i,j}-\gamma_1 s^{\alpha_1-1}\widehat{w_0}(\pi)_{i,j}+&\cdots \\
	&\hspace{-9cm}\cdots+\gamma_m s^{\alpha_m}\widetilde{\widehat{w}}(s,\pi)_{i,j}-\gamma_m s^{\alpha_m-1}\widehat{w_0}(\pi)_{i,j}+\pi_{i}^2\widetilde{\widehat{w}}(s,\pi)_{i,j}=0, \quad s>0, \\
\widehat{w}(t,\pi)_{i,j}|_{_{_{t=0}}}&=\widehat{w_0}(\pi)_{i,j}.
	\end{align*}
	The latter implies that
	\[
	\widetilde{\widehat{w}}(s,\pi)_{i,j}=\frac{s^{\alpha_0-1}+\gamma_1 s^{\alpha_1-1}+\cdots+\gamma_m s^{\alpha_m-1}}{s^{\alpha_0}+\gamma_1 s^{\alpha_1}+\cdots+\gamma_m s^{\alpha_m}+\pi_{i}^2}\widehat{w_0}(\pi)_{i,j}, \quad s>0, 
	\]
	and by the inverse Laplace transform we obtain
	\begin{align*}
		\widehat{w}(t,\pi)_{i,j}=\sum_{k=0}^{m}\gamma_k t^{\alpha_0-\alpha_k}&E_{(\alpha_0-\alpha_1,\ldots,\alpha_0-\alpha_m,\alpha_0),\alpha_0-\alpha_k+1}(-\gamma_1 t^{\alpha_0-\alpha_1},\ldots \\
		&\hspace{3cm}\ldots,-\gamma_m t^{\alpha_0-\alpha_m},-\pi_{i}^2 t^{\alpha_0})\widehat{w_0}(\pi)_{i,j},
	\end{align*}
	with $\gamma_0=1.$ Arguing as before, we see that the solution to the Cauchy problem \eqref{multi-heat} is given by 
 \begin{equation*}
w(t,x)=\sum_{k=0}^{m}t^{\alpha_0-\alpha_k}E_{(\alpha_0-\alpha_1,\ldots,\alpha_0-\alpha_m,\alpha_0),\alpha_0-\alpha_k+1}(-\gamma_1 t^{\alpha_0-\alpha_1},\ldots,-\gamma_m t^{\alpha_0-\alpha_m},-t^{\alpha_0}\mathcal{R}) w_0(x)\label{multi-solution}.
		\end{equation*}
 On the other hand, by the estimates in \cite[Lemma 3.3]{karel} (\cite[Lemma 3.2]{multi-estimate} or \cite[Lemma 3]{otra-estimacion}) we have
	\begin{align*}
		(1+\pi_{i}^2)^{s/\nu}&|\widehat{w}(t,\pi)_{i,j}|\leqslant \sum_{k=0}^{m}\gamma_k t^{\alpha_0-\alpha_k}|E_{(\alpha_0-\alpha_1,\ldots,\alpha_0-\alpha_m,\alpha_0),\alpha_0-\alpha_k+1}(-\gamma_1 t^{\alpha_0-\alpha_1},\ldots \nonumber\\
		&\hspace{3cm}\ldots,-\gamma_m t^{\alpha_0-\alpha_m},-\pi_{i}^2 t^{\alpha_0})|(1+\pi_{i}^2)^{s/\nu}|\widehat{w_0}(\pi)_{i,j}| \nonumber\\
		&\leqslant C_{T,\vec{\alpha},\vec{\gamma}}\left(\sum_{k=0}^{m}\gamma_k t^{\alpha_0-\alpha_k}\right)\frac{(1+\pi_{i}^2)^{s/\nu}}{1+\pi_{i}^2 t^{\alpha_0}}|\widehat{w_0}(\pi)_{i,j}| \\
  &\leqslant
\left\{
\begin{array}{rccl}
&\displaystyle C_{T,\vec{\alpha},\vec{\gamma}}\left(\sum_{k=0}^{m}\gamma_k t^{\alpha_0-\alpha_k}\right)(1+\pi_{i}^2)^{s/\nu}|\widehat{w_0}(\pi)_{i,j}|,\quad s\in\mathbb{R},\nonumber \\
&\displaystyle C_{T,\vec{\alpha},\vec{\gamma}}\left(\sum_{k=0}^{m}\gamma_k t^{\alpha_0-\alpha_k}\right)\big(1+\pi_{i}^{2\left(\frac{s}{\nu}-1\right)}t^{-\alpha_0}\big)|\widehat{w_0}(\pi)_{i,j}|,\quad s\geqslant \nu,\nonumber \\
\end{array}
\right.
	\end{align*}
which implies by the Plancherel formula that
\begin{align*}
\|w(t,\cdot)\|_{L^{2}_{s}(G)}&\leqslant  \left\{
\begin{array}{rccl}
&\displaystyle C_{T,\vec{\alpha},\vec{\gamma}}\left(\sum_{k=0}^{m}\gamma_k t^{\alpha_0-\alpha_k}\right)\|w_0\|_{L^{2}_{s}(G)},\,\, s\in\mathbb{R}, \\
&\displaystyle C_{T,\vec{\alpha},\vec{\gamma}}\left(\sum_{k=0}^{m}\gamma_k t^{\alpha_0-\alpha_k}\right)(1+t^{-\alpha_0})\|w_0\|_{L^{2}_{s-\nu}(G)}, \,\, s\geqslant \nu,
\end{array}
\right.
\end{align*}
which completes the proof.
\end{proof}

\subsection{Multi-term $\mathcal{R}$-wave type equations}

Consider the following equation:  
\begin{equation}\label{Multi-WeatTypeEquationG}
	\left\{ \begin{aligned}
		\prescript{C}{}\partial_{t}^{\alpha_0}w(t,x)+\gamma_1\prescript{C}{}\partial_{t}^{\alpha_1}w(t,x)+\cdots+\gamma_m\prescript{C}{}\partial_{t}^{\alpha_m}w(t,x)+\mathcal{R}w(t,x)&=0,\,\,  \\
		w(t,x)|_{_{_{t=0}}}&=w_0(x), \\
  \partial_t w(t,x)|_{_{_{t=0}}}&=w_1(x)\,,
	\end{aligned}
	\right.
\end{equation}
for $0<t\leqslant T<+\infty$ and $x\in G$, where $\mathcal{R}$ is a positive Rockland operator of homogeneous degree $\nu$ on $G$, $w_0,w_1$ will be chosen in some suitable Sobolev spaces, $\gamma_i>0$ $(i=1,\ldots,m)$ and $0<\alpha_m<\alpha_{m-1}<\cdots<\alpha_1<\alpha_0$ where $1<\alpha_0<2$.

\begin{thm}\label{thm.multi-wave}
Let $G$ be a graded Lie group, $\mathcal{R}$ be a positive Rockland operator on $G$ of homogeneous degree $\nu$,  $0<\alpha_m<\alpha_{m-1}<\cdots<\alpha_1<\alpha_0$ and $1<\alpha_0<2$. Then there exists a unique continuous solution to equation \eqref{Multi-WeatTypeEquationG} given by
\begin{align*}
&w(t,x)=\sum_{k=0}^{m}t^{\alpha_0-\alpha_k}E_{(\alpha_0-\alpha_1,\ldots,\alpha_0-\alpha_m,\alpha_0),\alpha_0-\alpha_k+1}(-\gamma_1 t^{\alpha_0-\alpha_1},\ldots,-\gamma_m t^{\alpha_0-\alpha_m},-t^{\alpha_0}\mathcal{R}) w_0(x) \\
&+\sum_{k=0}^{m}t^{\alpha_0-\alpha_k+1}E_{(\alpha_0-\alpha_1,\ldots,\alpha_0-\alpha_m,\alpha_0),\alpha_0-\alpha_k+2}(-\gamma_1 t^{\alpha_0-\alpha_1},\ldots,-\gamma_m t^{\alpha_0-\alpha_m},-t^{\alpha_0}\mathcal{R})w_1(x),
\end{align*}
for any $0<t\leqslant T$ and $x\in G.$  
Moreover, we have the following Sobolev norm estimates:
\begin{enumerate}
    \item For any $s\in\mathbb{R}$ and $w_0,w_1\in L^{2}_{s}(G)$ we have 
     \[
     \|w(t,\cdot)\|_{L^{2}_{s}(G)}\leqslant C_{T,\vec{\alpha},\vec{\gamma}}\left(\sum_{k=0}^{m}\gamma_k t^{\alpha_0-\alpha_k}\right)\left(\|w_0\|_{L^{2}_{s}(G)}+t\|w_1\|_{L^{2}_{s}(G)}\right),\quad 0<t\leqslant T.
     \]
     \item For any $s\geqslant \frac{\nu}{\alpha_0}$ and for $(w_0,w_1)\in L^{2}_{s}(G)\times L^{2}_{s-\nu/\alpha_0}(G)$ we have
     \[
    \|w(t,\cdot)\|_{L^{2}_{s}(G)}\leqslant C_{T,\vec{\alpha},\vec{\gamma}}\left(\sum_{k=0}^{m}\gamma_k t^{\alpha_0-\alpha_k}\right)\left(\|w_0\|_{L^{2}_{s}(G)}+(1+t)\|w_1\|_{L^{2}_{s-\nu/\alpha_0}(G)}\right),\quad 0<t\leqslant T.
     \]
     \item For any $s \geqslant \nu$ and for $w_0,w_1 \in L^{2}_{s-\nu}(G)$ we have
     \[
       \|w(t,\cdot)\|_{L^{2}_{s}(G)}\leqslant C_{T,\vec{\alpha},\vec{\gamma}}\left(\sum_{k=0}^{m}\gamma_k t^{\alpha_0-\alpha_k}\right)(1+t^{-\alpha_0})\left(\|w_0\|_{L^{2}_{s-\nu}(G)}+t\|w_1\|_{L^{2}_{s-\nu}(G)}\right)\quad 0<t\leqslant T.
     \]
\end{enumerate}

\end{thm}
\begin{proof}
An application of the group Fourier transform to the equation \eqref{Multi-WeatTypeEquationG} gives 
By applying the group Fourier transform to the equation \eqref{Multi-HeatTypeEquationG} we obtain
\begin{equation*}
\left\{ \begin{split}
\prescript{C}{}\partial_{t}^{\alpha_0}\widehat{w}(t,\pi)+\gamma_1\prescript{C}{}\partial_{t}^{\alpha_1}\widehat{w}(t,\pi)+\cdots+\gamma_m\prescript{C}{}\partial_{t}^{\alpha_m}\widehat{w}(t,\pi)+\pi(\mathcal{R})\widehat{w}(t,\pi)&=0,\quad \pi\in\widehat{G}, \\
\widehat{w}(t,\pi)|_{_{_{t=0}}}&=\widehat{w_0}(\pi) \\
\partial_t \widehat{w}(t,\pi)|_{_{_{t=0}}}&=\widehat{w_1}(\pi)\,,
\end{split}
\right.
\end{equation*}
which in turn implies
\begin{equation*}
\left\{ \begin{split}
\prescript{C}{}\partial_{t}^{\alpha_0}\widehat{w}(t,\pi)_{i,j}+\gamma_1\prescript{C}{}\partial_{t}^{\alpha_1}\widehat{w}(t,\pi)_{i,j}+\cdots+\gamma_m\prescript{C}{}\partial_{t}^{\alpha_m}\widehat{w}(t,\pi)_{i,j}+\pi(\mathcal{R})\widehat{w}(t,\pi)_{i,j}&=0, \\
\widehat{w}(t,\pi)_{i,j}|_{_{_{t=0}}}&=\widehat{w_0}(\pi)_{i,j} \\
 \partial_t \widehat{w}(t,\pi)_{i,j}|_{_{_{t=0}}}&=\widehat{w_1}(\pi)_{i,j}\,,
\end{split}
\right.
\end{equation*}
for all $(i,j)$ and for all $\pi \in \widehat{G}$. An application of the Laplace Fourier transform in $t$ gives 
\[
	\widetilde{\widehat{w}}(s,\pi)_{i,j}=\frac{s^{\alpha_0-1}+\gamma_1 s^{\alpha_1-1}+\cdots+\gamma_m s^{\alpha_m-1}}{s^{\alpha_0}+\gamma_1 s^{\alpha_1}+\cdots+\gamma_m s^{\alpha_m}+\pi_{i}^2}\left[\widehat{w_0}(\pi)_{i,j}+s^{-1}\widehat{w_1}(\pi)_{i,j}\right], \quad s>0\,.
	\]
 Thus applying the inverse Laplace transform and arguing as we did in Theorem \ref{multi-heat-thm}, we see that the solution to the Cauchy problem \eqref{Multi-WeatTypeEquationG} is given explicitly by 
 \begin{align*}
&w(t,x)=\sum_{k=0}^{m}t^{\alpha_0-\alpha_k}E_{(\alpha_0-\alpha_1,\ldots,\alpha_0-\alpha_m,\alpha_0),\alpha_0-\alpha_k+1}(-\gamma_1 t^{\alpha_0-\alpha_1},\ldots,-\gamma_m t^{\alpha_0-\alpha_m},-t^{\alpha_0}\mathcal{R}) w_0(x) \\
&+\sum_{k=0}^{m}t^{\alpha_0-\alpha_k+1}E_{(\alpha_0-\alpha_1,\ldots,\alpha_0-\alpha_m,\alpha_0),\alpha_0-\alpha_k+2}(-\gamma_1 t^{\alpha_0-\alpha_1},\ldots,-\gamma_m t^{\alpha_0-\alpha_m},-t^{\alpha_0}\mathcal{R})w_1(x)\,.
\end{align*}
Moreover, estimating as above we get 
\begin{align*}
		&(1+\pi_{i}^2)^{s/\nu}|\widehat{w}(t,\pi)_{i,j}|\leqslant  C_{T,\vec{\alpha},\vec{\gamma}}\left(\sum_{k=0}^{m}\gamma_k t^{\alpha_0-\alpha_k}\right)\frac{(1+\pi_{i}^2)^{s/\nu}}{1+\pi_{i}^2 t^{\alpha_0}}\left[|\widehat{w_0}(\pi)_{i,j}|+t|\widehat{w_1}(\pi)_{i,j}|\right] \\
  &\leqslant
\left\{
\begin{array}{rccl}
&\displaystyle C_{T,\vec{\alpha},\vec{\gamma}}\left(\sum_{k=0}^{m}\gamma_k t^{\alpha_0-\alpha_k}\right)(1+\pi_{i}^2)^{s/\nu}\left[|\widehat{w_0}(\pi)_{i,j}|+t|\widehat{w_1}(\pi)_{i,j}|\right],\,\, s\in\mathbb{R},\nonumber \\
    & \displaystyle C_{T,\vec{\alpha},\vec{\gamma}}\left(\sum_{k=0}^{m}\gamma_k t^{\alpha_0-\alpha_k}\right)(1+\pi_{i}^2)^{s/\nu} |\widehat{w_0}(\pi)_{i,j}|+\big(t+\pi_{i}^{2\left(\frac{s}{\nu}-\frac{1}{\alpha_0} \right)}\big)|\widehat{w_1}(\pi)_{i,j}|,\,\, s\geqslant \frac{\nu}{\alpha_0}, \nonumber \\
&\displaystyle C_{T,\vec{\alpha},\vec{\gamma}}\left(\sum_{k=0}^{m}\gamma_k t^{\alpha_0-\alpha_k}\right)\big(1+\pi_{i}^{2\left(\frac{s}{\nu}-1\right)}t^{-\alpha_0}\big)\left[|\widehat{w_0}(\pi)_{i,j}|+t|\widehat{w_1}(\pi)_{i,j}|\right] ,\,\, s\geqslant \nu,\nonumber\, \\
\end{array}
\right.
	\end{align*}
 where for the second term of the second estimate we have considered the supremum of the function $h(t)=\frac{t}{1+\pi_{i}^{2}t^{\alpha_0}}$.
Summation of the above over $(i,j)$ and integration over $\widehat{G}$, followed by an application of the Plancherel fomrula \eqref{plan.gr} yields
\begin{align*}
\|w(t,\cdot)\|_{L^{2}_{s}(G)}&\lesssim  
\left\{
\begin{array}{rccl}
&\displaystyle C_{T,\vec{\alpha},\vec{\gamma}}\left(\sum_{k=0}^{m}\gamma_k t^{\alpha_0-\alpha_k}\right)\left(\|w_0\|_{L^{2}_{s}(G)}+t\|w_1\|_{L^{2}_{s}(G)}\right),\quad s\in\mathbb{R}, \\
&\displaystyle C_{T,\vec{\alpha},\vec{\gamma}}\left(\sum_{k=0}^{m}\gamma_k t^{\alpha_0-\alpha_k}\right)\left(\|w_0\|_{L^{2}_{s}(G)}+(1+t)\|w_1\|_{L^{2}_{s-\nu/\alpha_0}(G)}\right),\quad s\geqslant \frac{\nu}{\alpha_0},\nonumber \\
&\displaystyle C_{T,\vec{\alpha},\vec{\gamma}}\left(\sum_{k=0}^{m}\gamma_k t^{\alpha_0-\alpha_k}\right)(1+t^{-\alpha_0})\left(\|w_0\|_{L^{2}_{s-\nu}(G)}+t\|w_1\|_{L^{2}_{s-\nu}(G)}\right), \quad s\geqslant \nu,
\end{array}
\right.
\end{align*}
and we have proved Theorem \ref{thm.multi-wave}.
 \end{proof}
 \section{Examples}\label{examples}
 In this section we provide the reader with examples in the parts \ref{itm:a} and \ref{itm:a1} below of positive self-adjoint operators on a Hilbert space $L^2(X)$ where $X=\mathbb{R}^d$ is a Euclidean space, or $X=M$ a manifold, respectively, and in the part \ref{itm:b} of settings of graded Lie groups. In the latter case we additionally consider differential operators in these settings.

 \begin{enumerate}[label=(\alph*)]
     \item \label{itm:a} \textbf{Self-adjoint operators on Euclidean spaces}
     \begin{itemize}
         \item \textbf{Harmonic oscillator.} For any dimension $d\geqslant 1$, we consider the harmonic oscillator 
     \[
     \mathcal{L}:=-\Delta+|x|^2\,,\quad x \in \mathbb{R}^d\,,
     \]
     where $\Delta$ stands for the Laplace operator. The operator $\mathcal{L}$ is essentially self-adjoint and is densley defined on the Hilbert space $L^2(\mathbb{R}^d)$. Its eigenfunctions are the well-known Hermite functions 
     which are an orthonormal basis of $L^2(\mathbb{R}^d)$ and the corresponding eigenvalues are known as well. 
     \item \textbf{Anharmonic oscillator.} For any dimension $d\geqslant 1$, the anharmonic oscillators can take the form 
     \[
     \mathcal{L}:=q(D)+p(x)\,,\quad x \in \mathbb{R}^d\,,
     \]
     where $q,p: \mathbb{R}^d \rightarrow \mathbb{R}$ are certain polynomial of even degree, see \cite{CDR21,CDR21-2} for a detailed description of anharmonic oscillators of any dimension. A well-studied example of an anharmonic oscillator is given by the formula
     \[
     \mathcal{L}=(-\Delta)^k+|x|^{\ell}\,,
     \]
     where $k,\ell \geqslant 1$ are integers.
     \item \textbf{Landau Hamiltonian in 2 dimensions.} The Landau Hamiltonian is an important example that is classical in physics. In $2D$ it is given by 
     \[
     \mathcal{L}:=\frac{1}{2} \left( \left(i \frac{\partial}{\partial x}-By\right)^2+\left(i \frac{\partial}{\partial y}+Bx \right)^2 \right)\,,
     \]
     where $B>0$ is a positive constant. The operator $\mathcal{L}$ is a self-adjoint operator on $L^2(\mathbb{R}^2)$. Its spectrum is well-known and its eigenvalues have infinite multiplicity, see \cite{F28}, \cite{L30}.
     \item \textbf{Sturm–Liouville equations as self-adjoint differential operators.} We now look at the Sturm–Liouville form in the weighted Hilbert space $L^2([a,b],w(x)\rm{d}x)$ given by the following differential equation:
     \[
     \mathcal{L}(u)=-\lambda w(x)u(x),\quad x\in[a,b],
     \]
     where 
     \[
     \mathcal{L}(u)=\frac{1}{w(x)}\left(\frac{\rm{d}}{\rm{d}x}\bigg[r(x)u^{\prime}(x)\bigg]+s(x)u(x)\right),
     \]
     for some given functions $w,r$ and $s.$ Let us recall that in the considered space the scalar product is 
     \[
     \langle f,g \rangle=\int_a^b \overline{f(x)}g(x)w(x){\rm d}x.
     \]
     Under this product and since $\mathcal{L}$ is defined for sufficiently smooth functions which satisfy good-enough regular boundary conditions, it can be shown that the operator $\mathcal{L}$ is self-adjoint. To get a specific discrete spectrum for this operator, we can just choose some specific given functions.  
     \item \textbf{Second order differential operator with involution.} Let us consider the following differential operator in $L^2(0,\pi)$ given by 
     \[
     \mathcal{L}(v)=v^{\prime\prime}(x)-\epsilon v^{\prime\prime}(\pi-x),\quad 0<x<\pi,\quad\text{for some}\quad |\epsilon|<1,
     \]
     having involution operator in the same space along with the Dirichlet conditions 
     \[
     v(0)=0,\quad \text{and}\quad v(\pi)=0.
     \]
    The operator $\mathcal{L}$ is self-adjoint and the eigenvalues (respectively the eigenfunctions) can be found explicitly for any $|\epsilon|<1$, see e.g. \cite{TT}.  
     \end{itemize}
     \item \label{itm:a1} \textbf{Operators on manifolds.}
     \begin{itemize}
         \item \textbf{Operator with periodic boundary conditions.} Let $M=\overline{\Omega}$ where $\Omega=(0,1)^n$ is a compact manifold. We define the operator
         \[
         \mathcal{L}:=\sum_{j=1}^{n}\frac{\partial}{\partial x_{j}^{2}}\,,
         \]
         on $M$ together with the boundary conditions:
         \begin{equation}\label{BC}
         f(x)|_{x_j=0}= f(x)|_{x_j=1}\,,\quad \text{and}\quad \frac{\partial f}{\partial x_j}(x)|_{x_j=0}=\frac{\partial f}{\partial x_j}(x)|_{x_j=1}\,,\ j=1,\cdots,n\,.
         \end{equation}
         The domain of $\mathcal{L}$ is given by
         \[
         \textnormal{D}(\mathcal{L}):=\{f \in L^2(\Omega)\,:\,\mathcal{L}f \in L^2(\Omega)\,:\, f\quad \text{satisfies}\quad \eqref{BC} \}\,.
         \]
         The operator $\mathcal{L}$ with the boundary conditions given by \eqref{BC} is self-adjoint, and its eigenfunctions are explicitly known;  see \cite{[15]}.
         
         \item \textbf{Operator with non-periodic boundary conditions.} Let $M=(0,1)$ and let 
         \[
         \mathcal{L}:=-i \frac{d}{dx}\,,
         \]
         be an operator on $M$ with domain 
         \[
         \textnormal{D}(\mathcal{L}):=\left\{f \in W_{1}^{2}[0,1]\,:\,af(0)+bf(1)+\int_{0}^{1}f(x)g(x)\,dx=0 \right\}\,,
         \]
         where $a,b \neq 0$, and $q \in C^{1}[0,1]$. The operator $\mathcal{L}$ is self-adjoint and has discrete spectrum. The complete spectral analysis of the operator $\mathcal{L}$ is explicitly known; see \cite{[15]}.
         \item \textbf{Elliptic pseudo differential operators on closed manifolds.} Let $M$ be a compact manifold without boundary (closed manifold), and let $\mathcal{L}$ be a positive elliptic pseudo-differential operator. For this operator, we know that it has discrete spectrum. For more discussions,  see e.g. \cite{closed}.  
     \end{itemize}

     \item \label{itm:b}  \textbf{Graded Lie groups} 
    \begin{itemize}
        \item \textbf{Heisenberg group $\mathbb{H}^n$.} The Heisenberg group $\mathbb{H}^n$ is the manifold $\mathbb{R}^{2n+1}$ endowed with the composition law
        \[
        (x,y,t)(x',y,t'):=(x+x',y+y',t+t'+\frac{1}{2}(xy'-x'y))\,,
        \]
        where $(x,y,t),(x',y,t') \in \mathbb{H}^n$. The canonical basis of its Lie algebra is given by the (left-invariant) vector fields
        \[
        X_j=\partial_{x_j}-\frac{y_j}{2}\partial_t\,,\quad Y_j=\partial_{y_j}+\frac{x_j}{2}\partial_t\,,\quad \text{and}\quad T=\partial_t\,.
        \]
        The sub-Laplacian in this setting is the simplest example of a Rockland operator and is given by
        \[
        \mathcal{R}_{sub}=-\sum_{j=1}^{n}(X_{j}^{2}+Y_{j}^{2})\,.
        \]
        More generally a positive Rockland operator in this setting is given by 
        \[
        \mathcal{R}=\sum_{i=1}^{2n+1}(-1)^{\frac{\nu_0}{\nu_i}}c_i A_{i}^{2\frac{\nu_0}{\nu_i}}\,,\quad c_i>0\,,
        \]
        where $\nu_i \in \{1,2\}$, $\nu_0$ is an even number, and $A_i \in \{X_j,Y_j,T \} $, for all $j=1,\cdots, n$ and $i=1,\cdots, 2n+1$. For the homogeneous dimension $Q$ we have $Q=2n+2$.
        \item \textbf{Engel group $\mathcal{B}_4$.} The Engel group $\mathcal{B}^4$ is the manifold $\mathbb{R}^{4}$ endowed with the composition law
       \begin{eqnarray*}
\lefteqn{(x_1,x_2,x_3,x_4) \times (y_1,y_2,y_3,y_4)}\\
&:=& (x_1+y_1,x_2+y_2,x_3+y_3-x_1y_2,x_4+y_4+\frac{1}{2}x_1^2y_2-x_1y_3)\,.
\end{eqnarray*}
         The canonical basis of its Lie algebra is given by the (left-invariant) vector fields
      \begin{equation*}\label{left.inv.eng}
\begin{split}
X_1(x)&= \frac{\partial}{\partial x_1}\,, \quad X_2(x)=\frac{\partial}{\partial x_2}-x_1 \frac{\partial}{\partial x_3}+ \frac{x^{2}_{1}}{2}\frac{\partial}{\partial x_4}\,, \\
            X_3(x)&= \frac{\partial}{\partial x_3}-x_1 \frac{\partial}{\partial x_4}\,,\quad X_4(x)= \frac{\partial}{\partial x_4}\,,
\end{split}
\end{equation*}
for $x=(x_1,x_2,x_3,x_4) \in \mathcal{B}^4=\mathbb{R}^4$. 
        The sub-Laplacian in this setting is the simplest example of a Rockland operator and is given by
        \[
        \mathcal{R}_{sub}=-\sum_{j=1}^{n}(X_{1}^{2}+X_{2}^{2})\,.
        \]
        More generally a positive Rockland operator in this setting is given by 
        \[
        \mathcal{R}=\sum_{i=1}^{4}(-1)^{\frac{\nu_0}{\nu_i}}c_i X_{i}^{2\frac{\nu_0}{\nu_i}}\,,\quad c_i>0,
        \]
        where $\nu_i \in \{1,2,3\}$ and $\nu_0$ is an even number that is a multiple of $3$.  For the homogeneous dimension $Q$ we have $Q=7$. For a detailed description of the Engel group we refer to \cite{Cha.a} and \cite{Cha.b}.
    \end{itemize}
 
 \end{enumerate}
 \begin{rem}
     Let us point out that in both examples in case \ref{itm:b} we know from \cite{RR18} that the trace of the spectral projections $E_{(0,s)}(\mathcal{R})$ has the following asymptotic behavior for any choice of Rockland operator of homogeneous degree $\nu$
\[
\tau\big(E_{(0,s)}(\mathcal{R})\big)\lesssim s^{Q/\nu},\quad s\to+\infty\,,
\]
 where $Q$ denoted the homogeneous dimension of the group. Hence, our consideration below satisfy the hypothesis \eqref{asymtotic-trace} as in Theorem \ref{Main-heat}. See also Example \ref{ex.graded}.
 \end{rem}

\section{Acknowledgements}
The authors were supported by the FWO Odysseus 1 grant G.0H94.18N: Analysis and Partial Differential Equations and by the Methusalem programme of the Ghent University Special Research Fund (BOF) (Grant number 01M01021). Marianna Chatzakou is supported by the FWO Fellowship grant No 12B1223N. Michael Ruzhansky is also supported by the EPSRC grant EP/R003025/2 and FWO Senior Research Grant G011522N.

\end{document}